\documentclass[MSNbibl,number,seceqn,citesort,dvips]{arxbj}
\usepackage{upgreek}
\usepackage{dcolumn}
\usepackage{accents}
\usepackage[left]{vmkcol}
\usepackage{graphicx}

\aid{0}
\volume{20}
\issue{1}
\pubyear{2014}
\firstpage{28}
\lastpage{77}
\doi{10.3150/12-BEJ475} 

\makeatletter

\renewcommand{\overrightarrow}[1]{\accentset{\longrightarrow}{#1}}

\newcolumntype{d}[1]{D{.}{.}{#1}}

\newcommand{\rrvert}{\vert}
\newcommand{\llvert}{\vert}

\newcommand{\RMo}{\mathrm{o}}
\newcommand{\RMO}{\mathrm{O}}

\newcommand{\RMe}{\mathrm{e}}

\newcommand{\mrmd}{\,\mathrm{d}}
\newcommand{\mrmdd}{\mathrm{d}}

\newtheorem{theorem}{Theorem}
\newproclaim{condition}{Condition}

\newproclaim{definition}{Definition}
\newtheorem{lemma}[theorem]{Lemma}
\newtheorem{proposition}[theorem]{Proposition}
\newremark{remark}{Remark}

\makeatother

\begin{document}
\begin{frontmatter}

\title{Gaussian semiparametric estimates on the unit sphere}
\runtitle{Gaussian semiparametric estimates on the unit sphere}

\begin{aug}
\author[1,2]{\fnms{Claudio} \snm{Durastanti}\thanksref{1,2,e1}\ead[label=e1,mark]{durastan@mat.uniroma2.it}},
\author[3]{\fnms{Xiaohong} \snm{Lan}\thanksref{3}\ead[label=e2]{xhlan@ustc.edu.cn}}
\and\\
\author[1]{\fnms{Domenico} \snm{Marinucci}\corref{}\thanksref{1,e3}\ead[label=e3,mark]{marinucc@mat.uniroma2.it}}
\runauthor{C. Durastanti, X. Lan and D. Marinucci} 
\address[1]{University of Rome Tor Vergata, Via della Ricerca Scientifica, 1, 00133 Roma, Italy.\\ \printead{e1,e3}}
\address[2]{University of Pavia, Via Ferrata, 1, 27100 Pavia, Italy}
\address[3]{University of Science and Technology of China, Hefei,
Anhui, China 230026, and University of Connecticut, 3 North Hillside Road Unit 6079
Storrs, CT 06269-6079, USA. \mbox{\printead{e2}}}
\end{aug}

\received{\smonth{10} \syear{2011}}
\revised{\smonth{4} \syear{2012}}

%
\begin{abstract}
We study the weak convergence (in the high-frequency limit) of the
parameter estimators of power spectrum coefficients associated with
Gaussian, spherical and isotropic random fields. In particular, we
introduce a Whittle-type approximate maximum likelihood estimator and
we investigate its asympotic weak consistency and Gaussianity, in both
parametric and semiparametric cases.
\end{abstract}

%

\end{frontmatter}

\section{Introduction}\label{SecIntroduction}

The purpose of this paper is to investigate the asymptotic behavior of a
Whittle-like approximate maximum likelihood procedure for the
estimation of
the spectral parameters (e.g., the \textit{spectral index}) of isotropic
Gaussian random fields defined on the unit sphere $\mathbb{S}^{2}$. In our
approach, we consider the expansion of the field into spherical harmonics,
that is, we implement a form of Fourier analysis on the sphere, and we implement
approximate maximum likelihood estimates under both parametric and
semiparametric assumptions on the behavior of the angular power
spectrum. We
stress that the asymptotic framework we are considering here is rather
different from usual -- in particular, we assume we are observing a single
realization of an isotropic field, the asymptotics being with respect to
higher and higher resolution data becoming available (i.e., higher and
higher frequency components being observed). In some sense, then the issues
we are considering are related to the growing area of fixed-domain
asymptotics (see, e.g., \cite{anderes,loh}). From the
point of
view of the proofs, on the other hand, our arguments are in some cases
reminiscent of those entertained, for instance, by \cite{Robinson}, where
semiparametric estimates of the long memory parameter for covariance
stationary processes are analyzed; see also \cite{guo} for related results
in the setting of anisotropic random fields.

In our assumptions, we do not impose a priori a parametric model on the
dependence structure of the random field we are analyzing; we rather
impose various forms of regularly varying conditions, which only
constrain the high-frequency behaviour of the angular power spectrum.
We are able to show consistency under the least restrictive
assumptions; a central limit theorem holds under more restrictive
conditions, while asymptotic Gaussianity can be established under
general conditions for a slightly-modified (\textit{narrow-band})
procedure, entailing a loss of a logarithmic factor in the rate of
convergence. Our analysis is strongly motivated by applications,
especially in a Cosmological framework (see, e.g.,
\cite{dode2004,cama}); in this area, huge datasets on isotropic,
spherical random fields (usually assumed to be Gaussian) are currently
being collected and made publicly available by celebrated satellite
missions such as \textit{WMAP} or \textit{Planck} (see, e.g.,
\url{http://map.gsfc.nasa.gov/}); parameter estimation of the spectral index
and other spectral parameters has been considered by many authors (see,
e.g., \cite{hamann} for a review), but no rigorous asymptotic
result has so far been produced, to the best of our knowledge. We thus
hope that the consistency and asymptotic Gaussianity properties we
provide for our Whittle-like procedure may provide a contribution
toward further developments. We refer also to
\cite{bkmpAoS,bkmpBer,fay08,glm,pbm06,pietrobon1,marpec2} for further
theoretical and applied results on angular power spectrum estimation,
in a purely nonparametric setting, and to
\cite{kimkoo,kookim,kim,kerkyphampic,dmg,leonenko1,leonenko2,ejslan,marpecbook}
for further results on statistical inference for spherical random
fields. Fixed-domain asymptotics for the tail behaviour of the spectral
density on Euclidean spaces has been recently considered also by
\cite{guo,anderes2} and \cite{wu}; the issue is of great interest, for
instance, in connection with kriging techniques for geophysical data
analysis, see \cite{stein} for a textbook reference.

The plan of the paper is as follows: in Section \ref
{subsecgaussrandomfiels}, we will recall briefly some well-known background
material on harmonic analysis for spherical isotropic random fields; in
Section \ref{secsphwhittle} we introduce Whittle-like maximum
pseudo-likelihood estimators for angular power spectrum coefficients based
on spherical harmonics; Section \ref{asconsistency} is devoted to the
asymptotic results, while in Section \ref{narrowband} we investigate
narrow-band estimates. The presence of observational noise is
considered in
Section \ref{noise}, while Section \ref{numerical} provides some numerical
evidence to validate the findings of the paper. Directions for future
research are discussed in Section \ref{conclusions}, while some auxiliary
technical results are collected in the \hyperref[app]{Appendix}.

\section{Spherical random fields and angular power spectrum} \label{subsecgaussrandomfiels}

In this section, we will present some well-known background results
concerning harmonic analysis on the sphere. We shall focus on zero-mean,
isotropic Gaussian random fields $T\dvtx \mathbb{S}^{2}\times\Omega
\rightarrow
\mathbb{R}$. It is well known that such fields can be given a spectral
representation such that
%
\begin{eqnarray}
\label{dirfour} T ( x ) &=&\sum_{l\geq0}\sum
_{m=-l}^{l}a_{lm}Y_{lm} ( x ) =
\sum_{l\geq0}T_{l} ( x ),
\\
\label{invfour} a_{lm} &=&\int_{\mathbb{S}^{2}}T ( x )
\overline{Y}_{lm} ( x ) \mrmd x,
\end{eqnarray}
where the set of homogenous polynomials $ \{ Y_{lm}\dvtx l\geq
0,m=-l,\ldots,l \} $ represents an orthonormal basis for the space
$%
L^{2} ( \mathbb{S}^{2},\mrmdd x )$, the class of functions defined on
the unitary sphere which are square-integrable with respect to the
measure $%
\mrmdd x$ (see, e.g., \cite{steinweiss,leonenko1,marpecbook},
for more details, and \cite{leosa,mal} for extensions). Note that
this equality holds in both $L^{2} ( \mathbb{S}^{2}\times\Omega,
\mrmdd x\otimes\mathbb{P} ) $ and $L^{2} ( \mathbb{P} ) $ senses
for every fixed $x\in\mathbb{S}^{2}$. We recall also that a field
$T (
\cdot ) $ is isotropic if and only if for every $g\in SO (
3 ) $ (the special group of rotations in $\mathbb{R}^{3}$) and
$x\in
\mathbb{S}^{2}$ (the unit sphere), we have
\[
T ( x ) \stackrel{d} {=}T ( gx ),
\]
where the equality holds in the sense of processes.

An explicit form for spherical harmonics is given in spherical
coordinates $%
\vartheta\in [ 0,\uppi  ]$, $\varphi\in[ 0,2\uppi  )$
by:
\begin{eqnarray*}
Y_{lm} ( \vartheta,\varphi ) &=&\sqrt{\frac{2l+1}{4\uppi }
\frac
{%
( l-m ) !}{ ( l+m ) !}}P_{lm} ( \cos\vartheta ) \RMe^{\mathrm{i}m\varphi}\qquad
\mbox{for }m\geq0,
\\
Y_{lm} ( \vartheta,\varphi ) &=& ( -1 )^{m}
\overline{Y}%
_{l,-m} ( \vartheta,\varphi )\qquad\mbox{for }m<0,
\end{eqnarray*}
$P_{lm} ( \cos\vartheta ) $ denoting the associated Legendre
function; for $m=0$, we have $P_{l0} ( \cos\vartheta )
=P_{l} ( \cos\vartheta )$, the standard set of Legendre
polynomials (see again \cite{steinweiss,marpecbook}). The following
orthonormality property holds:%
\[
\int_{\mathbb{S}^{2}}Y_{lm} ( x ) \overline{Y}_{l^{\prime
}m^{\prime}}
( x ) \mrmd x=\delta_{l}^{l^{\prime}}\delta_{m}^{m^{\prime}}.
\]

For an isotropic Gaussian field, the spherical harmonics coefficients $%
a_{lm} $ are Gaussian complex random variables such that%
\[
\mathbb{E} ( a_{lm} ) =0,\qquad \mathbb{E} ( a_{l_{1}m_{2}}%
\overline{a}_{l_{2}m_{2}} ) =\delta_{l_{2}}^{l_{1}}
\delta_{m_{2}}^{m_{1}}C_{l},
\]
where of course the angular power spectrum $C_{l}$ fully characterizes the
dependence structure under Gaussianity; here, $\delta_{a}^{b}$ is the
Kronecker delta, taking value one for $a=b$, zero otherwise. Further
characterizations of the spherical harmonics coefficients are provided, for
instance, by \cite{bm,marpecbook}; here we simply recall that
\[
\frac{a_{l0}^{2}}{C_{l}}\sim\chi_{1}^{2}\qquad\mbox{for }m=0,\qquad
\frac{%
2\llvert  a_{lm}\rrvert^{2}}{C_{l}}\sim\chi_{2}^{2}\qquad\mbox{for }%
m=\pm1,\pm2,\ldots,\pm l,
\]
where all these random variables are independent. Given a realization
of the
random field, an estimator of the angular power spectrum can be defined
as:%
%
\begin{equation}
\label{powest} \widehat{C}_{l}=\frac{1}{2l+1}\sum
_{m=-l}^{l}\llvert a_{lm}
\rrvert^{2},
\end{equation}
the so-called empirical angular power spectrum. It is immediately seen
that%
\[
\mathbb{E}\widehat{C}_{l}=\frac{1}{2l+1}\sum
_{m=-l}^{l}C_{l}=C_{l}, \qquad\operatorname{Var}
\biggl( \frac{\widehat{C}_{l}}{C_{l}} \biggr) =\frac
{2}{2l+1}\rightarrow0 \qquad\mbox{for
}l\rightarrow+\infty.
\]

We shall now focus on some semiparametric models on the angular power
spectrum; here, by semiparametric we mean that we shall assume a parametric
form on the asymptotic behavior of $C_{l}$, but we shall refrain from a full
characterization over all multipoles $l$. More precisely, we formulate the
following:
%
\begin{condition}
\label{A} The random field $T(x)$ is Gaussian and isotropic with angular
power spectrum such that:%
%
\begin{equation}
\label{reg2} C_{l}=G ( l ) l^{-\alpha_{0}}>0,
\end{equation}
where $\alpha_{0}>2$ and for all $l=1,2,\ldots$%
\[
0<c_{1}\leq G ( l ) \leq c_{2}<+\infty.
\]
\end{condition}

Condition \ref{A} seems very mild, as it is basically requiring only some
form of
regular variation on the tail behavior of the angular power spectrum
$C_{l}$. For instance, in the CMB framework the so-called
\textit{Sachs--Wolfe} power
spectrum (i.e., the leading model for fluctuations of the primordial
gravitational potential) takes the form (\ref{reg2}), the spectral
index $%
\alpha_{0}$ capturing the scale invariance properties of the field
itself ($%
\alpha_{0}$ is expected to be close to 2 from theoretical
considerations, a~prediction so far in good agreement with
observations, see, e.g., \cite{dode2004} and \cite{Larson}). For our
asymptotic results below, we shall need to strengthen it somewhat; as
we shall see, Condition \ref{A1} will turn out to be sufficient to
establish a rate of convergence for our estimator, under Condition
\ref{A2} we will be able to provide a Law of Large Numbers, while under
Condition \ref{A3} our estimates will be shown to be asymptotically
Gaussian and centered, thus making statistical inference feasible. On
the other hand, in Section \ref{narrowband} we shall be able to provide
narrow-band estimates with asymptotically centred limiting Gaussian law
under Condition \ref{A1}, to the price of a logarithmic term in the
rate of convergence. Of course, the conditions below are nested, that
is, Condition \ref{A3} implies Condition~\ref{A2}, which trivially
implies Condition~\ref{A1}.\looseness=1
%
\begin{condition}
\label{A1}Condition \ref{A} holds and moreover, $G(l)$ satisfies the
smoothness condition%
\[
G(l)=G_{0} \biggl\{ 1+\RMO\biggl(\frac{1}{l}\biggr) \biggr\}.
\]
\end{condition}
%
\begin{condition}
\label{A2} Condition \ref{A1} holds and moreover, $G(l)$ satisfies%
\[
G(l)=G_{0} \biggl\{ 1+\frac{\kappa}{l}+\RMo\biggl(\frac{1}{l}
\biggr) \biggr\}.
\]
\end{condition}
%
\begin{condition}
\label{A3} Condition \ref{A2} holds with $\kappa=0$, that is, $G(l)$ satisfies
the smoothness condition%
\[
G(l)=G_{0} \biggl\{ 1+\RMo\biggl(\frac{1}{l}\biggr) \biggr\}.
\]
\end{condition}

A straightforward example that satisfies the previous assumptions is
provided by the rational function%
%
\begin{equation}
\label{poly} G ( l ) =\frac{\Pi_{1} ( l ) }{\Pi_{2} ( l )
}=%
\frac{p_{k}l^{k}+\cdots+p_{1}l+p_{0}}{q_{k}l^{k}+\cdots+q_{1}l+pq},
\end{equation}
where $\Pi_{1}(l)$ and $\Pi_{2} ( l ) $ are positive valued
polynomials of order $k\in\mathbb{N}$, such that:
\[
0<c_{1}\leq\frac{\Pi_{1} ( l ) }{\Pi_{2} ( l )
}\leq c_{2}<+\infty.
\]
Clearly (\ref{poly}) satisfies Condition \ref{A2} (and hence Condition \ref{A1}) for
\[
G_{0}=\frac{p_{k}}{q_{k}}\quad\mbox{and}\quad\kappa=
\frac{p_{k-1}}{p_{k}}-\frac
{%
q_{k-1}}{q_{k}};
\]
Condition \ref{A3} is satisfied when $p_{k-1}=q_{k-1}=0$, or, more
generally, for $\frac{p_{k-1}}{p_{k}}=\frac{q_{k-1}}{q_{k}}$.

\section{A Whittle-like approximation to the likelihood function} \label{secsphwhittle}

Our aim in this section is to discuss heuristically a Whittle-like
approximation for the log-likelihood of isotropic spherical Gaussian fields,
and to derive the corresponding estimator. Assume that the triangular
array $%
\{ a_{lm} \} $, $m=-l,\ldots,l$, $l=1,2,\ldots,L$, is
evaluated from the
observed field $ \{ T(x) \} $, by means of (\ref{invfour}). Our
motivating rationale is the idea that a set of harmonic components up to
multipole $L$ can be reconstructed without observational noise or numerical
error, whereas the following are simply discarded; this is clearly a
simplified picture, but we believe it provides an accurate
approximation to
many current experimental set-ups. Of course, $L$ grows larger when more
sophisticated experiments are run ($L$ can be considered in the order of
$500/600$ for data collected from \textit{WMAP} and $1500/2000$ for those from
\textit{Planck}).  It is readily seen from (\ref{powest}) that%
\[
\widehat{C}_{l}=\frac{1}{2l+1} \Biggl\{ a_{l0}^{2}+2
\sum_{m=1}^{l} \bigl[ \Re \{
a_{lm} \} \bigr]^{2}+2\sum_{m=1}^{l}
\bigl[ \Im%
\{ a_{lm} \} \bigr]^{2} \Biggr\},
\]
where the variables $ \{ a_{l0},\sqrt{2}\Re \{ a_{l1} \},
\sqrt{2}\Im \{ a_{l1} \},\ldots,\sqrt{2}\Re \{
a_{ll} \},\sqrt{2}\Im \{ a_{ll} \}  \} $ are
i.i.d. Gaussian variables with law $\mathcal{N}(0,C_{l})$, see \cite{bm}.
The likelihood function can then be written down as
\[
-2\log\mathcal{L}_{l}\bigl(\theta; \{ a_{lm}
\}_{m=-l}^{l}\bigr)=\mathrm{const}+(2l+1)\frac{\widehat{C}_{l}}{C_{l}(\theta
)}-(2l+1)\log
\frac{\widehat{C}_{l}}{C_{l}(\theta)}.
\]

Clearly this landscape is overly simplified, for instance, due to numerical
errors and aliasing effects the expected value $\mathbb{E}\llvert
a_{lm}\rrvert^{2}$ may not be exactly equal to the population
model $C_{l}(\theta)$; however in Conditions \ref{A} and following we
are allowing
the two to differ to various degrees, and we expect this to cover to
some of
effect these experimental features that we are neglecting. Also, rather than
a sharp cutoff at $L$, a smooth transition toward noisier frequencies would
represent more efficiently actual experimental circumstances; we shall
address this issue later on in this paper. Finally, it may be unreasonable
to assume that the spherical surface is fully observed; for most
experimental set-ups, either in Cosmology or in Geophysics, only
subsets are
actually sampled. This problem can be addressed by focussing on wavelet
transforms rather than standard Fourier analysis; we shall consider this
extension in a different work.

An alternative heuristics for our framework can be introduced considering
that for $l=1,2,\ldots,L$, the following Fourier components can be
observed on
a discrete grid of points $ \{ x_{1},\ldots,x_{K} \} $%
\[
\overrightarrow{T}_{l}= \bigl\{ T_{l}(x_{1}),
\ldots,T_{l}(x_{k}),\ldots,T_{l}(x_{K})
\bigr\}.
\]
To simplify our discussion, we shall also pretend that $\mathcal{X}%
_{k}:= \{ x_{1},\ldots,x_{K} \} $ form a set of approximate
\textit{cubature points} with constant \textit{cubature weights} $\lambda_{k}=4\uppi /K$
(see, e.g., \cite{MNW,npw1}), so that we have%
\[
\sum_{k}\frac{4\uppi }{K}Y_{lm_{1}}(x_{k})
\overline {Y}_{lm_{2}}(x_{k})\simeq \delta_{m_{1}}^{m_{2}}
\qquad\mbox{for }l=1,2,\ldots,L.
\]
As discussed also by \cite{bkmpBer}, the number of cubature points must
grow at least as quickly as the square of the highest multipole
considered, that is, $L^{2}=\RMO(\operatorname{card}(\mathcal{X}_{k}))$. For instance,
for a satellite experiment such as \textit{Planck} the pixelization has
cardinality of order $5\times 10^{6}$, and the highest multipole that
can be analyzed correspond broadly to the order $l=2\times10^{3}$. As
before, this landscape is overly simplified; for instance, cubature
weights on the sphere are known not to be constant, but their variation
is usually considered numerically negligible.

The frequency components $T_{l}$ are well known to be independent and
we can
hence write down the likelihood function as
\[
\mathcal{L}(\theta;T):=\prod_{l=1}^{L}
\mathcal{L}_{l}(\theta;%
\overrightarrow{T}_{l}),
\]
where
\begin{eqnarray*}
\mathcal{L}_{l}(\theta;\overrightarrow{T}_{l})&=&(2\uppi
)^{-(2l+1)/2}\Omega_{l}^{-1/2}\exp \biggl\{ -
\frac{1}{2}\overrightarrow{T}_{l}^{\prime
}
\Omega_{l}^{-1}\overrightarrow{T}_{l} \biggr\},
\\
\{ \Omega_{l} \}_{jk}&=& \biggl\{ \Omega_{l}(x_{j},x_{k})=
\frac
{2l+1%
}{4\uppi }C_{l}P_{l}\bigl( \langle
x_{j},x_{k} \rangle\bigr) \biggr\}.
\end{eqnarray*}
The matrix $\Omega_{l}$ can be (approximately) decomposed as follows:%
\begin{eqnarray*}
\Omega_{l}&\simeq&\sqrt{\frac{4\uppi }{K}} \left[\matrix{
Y_{l,-l}(x_{1}) & Y_{l,-l+1}(x_{1}) &
\cdots& Y_{l,l}(x_{1})
\cr
Y_{l,-l}(x_{2})
& \cdots& \cdots& Y_{l,l}(x_{2})
\cr
\vdots& \vdots& \vdots&
\vdots
\cr
Y_{l,-l}(x_{K}) & Y_{l,-l+1}(x_{K})
& \cdots& Y_{l,l}(x_{K})} \right]
\\
&&{}\times\frac{K}{4\uppi }C_{l}I_{2l+1}\times\sqrt{
\frac{4\uppi }{K}} \left[ \matrix{ \overline{Y}_{l,-l}(x_{1})
& \overline{Y}_{l,-l}(x_{2}) & \cdots& \overline{Y}%
_{l,-l}(x_{K})
\cr
\overline{Y}_{l,-l+1}(x_{1}) & \cdots& \cdots&
\overline{Y}_{l,-l+1}(x_{K})
\cr
\vdots& \vdots& \vdots& \vdots
\cr
\overline{Y}_{l,l}(x_{1}) & \overline{Y}_{l,l}(x_{2})
& \cdots& \overline {Y}_{l,l}(x_{K})} \right]
\\
&=&\!\!:\mathcal{Y}_{l}\times C_{l}(\theta)I_{2l+1}
\times\mathcal {Y}_{l}^{\ast}.
\end{eqnarray*}
In fact
\[
\mathcal{Y}_{l}^{\ast}\mathcal{Y}_{l}\simeq
I_{2l+1}\quad\mbox{and}\quad\det \{ \Omega_{l} \} \simeq
C_{l}^{2l+1}(\theta).
\]
Hence,
\[
-2\log\mathcal{L}_{l}(\theta;\overrightarrow{T}_{l})\simeq
K+(2l+1)\log C_{l}(\theta)+ \bigl\{ \overrightarrow{T}_{l}^{\prime}
\mathcal {Y}_{l}\times C_{l}^{-1}(
\theta)I_{2l+1}\times\mathcal{Y}_{l}^{\ast}
\overrightarrow {T}%
_{l} \bigr\}.
\]
Now%
\begin{eqnarray*}
\mathcal{Y}_{l}^{\ast}\overrightarrow{T}_{l} &=&
\sqrt{\frac{4\uppi
}{K}} \left[ \matrix{ \overline{Y}_{l,-l}(x_{1})
& \overline{Y}_{l,-l}(x_{2}) & \cdots& \overline{Y}%
_{l,-l}(x_{K})
\cr
\overline{Y}_{l,-l+1}(x_{1}) & \cdots& \cdots&
\overline{Y}_{l,-l+1}(x_{K})
\cr
\vdots& \vdots& \vdots& \vdots
\cr
\overline{Y}_{l,l}(x_{1}) & \overline{Y}_{l,l}(x_{2})
& \cdots& \overline {Y}_{l,l}(x_{K})} \right]
\\
&&{}\times \left[ \matrix{ \sum_{m}a_{lm}Y_{lm}(x_{1})
\cr
\sum_{m}a_{lm}Y_{lm}(x_{2})
\cr
\vdots
\cr
\sum_{m}a_{lm}Y_{lm}(x_{K})}
\right]
\\
&=&\sqrt{\frac{4\uppi }{K}} \left[ \matrix{ \sum_{m_{1}}a_{lm_{1}}
\sum_{k}Y_{lm_{1}}(x_{k})\overline
{Y}_{l,-l}(x_{k})
\cr
\sum_{m_{1}}a_{lm_{1}}
\sum_{k}Y_{lm_{1}}(x_{k})
\overline{Y}_{l,-l+1}(x_{k})
\cr
\vdots
\cr
\sum
_{m_{1}}a_{lm_{1}}\sum_{k}Y_{lm_{1}}(x_{k})
\overline{Y}_{l,l}(x_{k})} \right] \\
&\simeq&\sqrt{
\frac{K}{4\uppi }} \left[ \matrix{ a_{l,-l}
\cr
a_{l,-l+1}
\cr
\vdots
\cr
a_{l,l}} \right],
\end{eqnarray*}
whence
\[
\biggl\{ \overrightarrow{T}_{l}^{\prime}\mathcal{Y}_{l}
\times\frac
{4\uppi }{K}%
\frac{1}{C_{l}(\theta)}I_{2l+1}\times
\mathcal{Y}_{l}^{\ast}%
\overrightarrow{T}_{l}
\biggr\} \simeq\sum_{m}\frac{\llvert
a_{lm}\rrvert^{2}}{C_{l}(\theta)}=(2l+1)
\frac{\widehat{C}_{l}}{%
C_{l}(\theta)}.
\]
As before, we can then conclude heuristically that%
%
\begin{equation}
\label{appwhit} -2\log\mathcal{L}_{l}(\theta;\overrightarrow{T}_{l})
\simeq \mathrm{const}+(2l+1)%
\frac{\widehat{C}_{l}}{C_{l}(\theta)}-(2l+1)\log\frac{\widehat
{C}_{l}}{%
C_{l}(\theta)}.
\end{equation}
Again we stress that for a general spherical random field with an
infinite-terms expansion such as (\ref{dirfour}) the relationship (\ref%
{appwhit}) cannot hold exactly; indeed, precise cubature formulae can be
established only for finite order spherical harmonics. In general, this may
introduce some numerical error: as mentioned before, however, we
pretend in
this paper that such correction factors are covered by Conditions
\ref{A}--\ref{A3}. In other words, we envisage a situation where data analysis is
carried over on multipoles $l$ where numerical errors are of smaller order
and the approximation (\ref{reg2}) holds for the expected variance of the
sample coefficients $ \{ a_{lm} \} $.

\section{Asymptotic results: Consistency and asymptotic Gaussianity}
\label{asconsistency}

As motivated in the \hyperref[SecIntroduction]{Introduction}, in this paper we shall not assume we have
actually available a fully parametric model for the angular power spectrum.
Instead, the idea will be to use an approximate maximum likelihood
estimator, which shall exploit the asymptotic approximation provided by
Condition \ref{A}, that is, $C_{l}\simeq Gl^{-\alpha}$. In view of the
discussion in the previous section, the following Definition seems rather
natural:
%
\begin{definition}
The \textup{Spherical Whittle estimator} for the parameters $(\alpha_{0},G_{0})$ is provided by%
\[
(\widehat{\alpha}_{L},\widehat{G}_{L}):=\arg
\min_{\alpha\in A,G\in
(0,\infty)}\sum_{l=1}^{L} \biggl\{
(2l+1)\frac{\widehat
{C}_{l}}{Gl^{-\alpha}}%
-(2l+1)\log\frac{\widehat{C}_{l}}{Gl^{-\alpha}} \biggr\}.
\]
\end{definition}
%
\begin{remark}
For general parametric models $C_{l}=C_{l}(\vartheta)$, the \textit{Spherical
Whittle estimator} for a parameter $\vartheta\in\Theta\subset\mathbb
{R}%
^{p}$ can be obviously defined as%
\[
\widehat{\vartheta}_{L}:=\arg\min_{\vartheta\in\Theta
}\sum
_{l=1}^{L} \biggl\{ (2l+1)\frac{\widehat{C}_{l}}{C_{l}(\vartheta)}%
-(2l+1)\log\frac{\widehat{C}_{l}}{C_{l}(\vartheta)} \biggr\}.
\]
\end{remark}
%
\begin{remark}
To ensure that the estimator exists, as usual we shall assume throughout
this paper that the parameter space for $\alpha$ is a compact subset
of $\mathbb{R}$; more precisely we take $\alpha\in A= [
a_{1},a_{2} ]$, $2<a_{1}<a_{2}<\infty$, and $G\in(0,\infty)$. This
is little more
than a formal requirement that\vadjust{\goodbreak} is standard in the literature on
(pseudo-)maximum likelihood estimation. It should be noted that Spherical
Whittle estimates are computationally extremely convenient, while their
counterpart in the real domain is for all practical purposes unfeasible,
given the dimension of current datasets.
\end{remark}
%
\begin{remark}
Under Condition \ref{A3}, it is readily seen that $(2l+1)\widehat{C}%
_{l}/Gl^{-\alpha_{0}}$ is asymptotically distributed as a Gamma random
variables of parameters $ \{ 2l+1,1 \} $, and the Spherical Whittle
estimator is asymptotically equivalent to exact maximum likelihood.
\end{remark}

We can rewrite in a more transparent form the previous estimator following
an argument analogous to \cite{Robinson}, that is, ``concentrating out''
the parameter $G$. Indeed, the previous
minimization problem is equivalent to let us consider%
\begin{eqnarray*}
(\widehat{\alpha}_{L},\widehat{G}_{L}) &:=&\arg
\min_{\alpha,G}\mathcal {R}%
_{L}(G,\alpha),
\\
\mathcal{R}_{L}(G,\alpha) &:=& \sum_{l=1}^{L}(2l+1)
\frac{\widehat
{C}_{l}}{%
Gl^{-\alpha}}+\sum_{l=1}^{L}(2l+1)\log
G\\
&&{}+\sum_{l=1}^{L}(2l+1)\log
l^{-\alpha}.
\end{eqnarray*}
Simple computations show that the minimization problem can be equivalently
reformulated as
%
\begin{eqnarray}
\label{defest} \widehat{\alpha}_{L} &=&\arg\min_{\alpha}R_{L}(
\alpha),
\nonumber\\[-8pt]\\[-8pt]
R_{L}(\alpha) &=& \Biggl( \log\widehat{G}(\alpha)-\frac{\alpha}{%
\sum_{l=1}^{L}(2l+1)}
\sum_{l=1}^{L}(2l+1)\log l \Biggr).
\nonumber
\end{eqnarray}
The proof of the following result is quite standard and goes largely along
the lines of an analogous results provided in \cite{Robinson}. As most of
the ones to follow, is delayed to the \hyperref[app]{Appendix}.
%
\begin{theorem}
\label{consistenza} Under Condition \ref{A}, as $L\rightarrow\infty$ we
have
\[
\widehat{\alpha}_{L}\rightarrow_{p}\alpha_{0};
\]
moreover, under Condition \ref{A1},
\[
\widehat{G}_{L}\rightarrow_{p}G_{0}.
\]
\end{theorem}

Next step is the investigation of the asymptotic distribution. To this aim,
we shall exploit some classical argument on asymptotic Gaussianity for
extremum estimates, as recalled, for instance, by \cite{NmcF}, Theorem 3.1.
%
\begin{theorem}
\label{clt0} Let $\widehat{\alpha}_{L}=\arg\min_{\alpha\in A}
R_{L}(\alpha)$ defined as in (\ref{defest}).

\begin{longlist}[(a)]
\item[(a)]
Under Condition \ref{A1} we have that%
%
\begin{equation}
\label{carave1}
\hspace*{-11pt}\bigl\{ \mathbb{E} ( \widehat{\alpha}_{L}-
\alpha_{0} )^{2} \bigr\}^{1/2}=\RMO\biggl(
\frac{\log L}{L}\biggr)\qquad\mbox{whence } ( \widehat{%
\alpha}_{L}-\alpha_{0} ) =\RMO_{p}\biggl(
\frac{\log L}{L}\biggr)\qquad\mbox{as }%
L\rightarrow\infty.
\end{equation}

\item[(b)] Under Condition \ref{A2} we have that%
%
\begin{equation}
\label{carave2} \frac{L}{4\log L} ( \widehat{\alpha}_{L}-
\alpha_{0} ) \longrightarrow_{p}-\kappa.
\end{equation}

\item[(c)] Under Condition \ref{A3} we have that
%
\begin{equation}
\label{carave3} \frac{\sqrt{2}L}{4} ( \widehat{\alpha}_{L}-
\alpha_{0} ) \stackrel{d%
} {\longrightarrow}\mathcal{N} (
0,1 ).
\end{equation}
\end{longlist}
\end{theorem}
\begin{pf}
We note first that under Condition \ref{A3}, (\ref{carave3}) is an immediate
consequence of (\ref{carave2}); on the other hand, the proof of (\ref%
{carave1}) follows on exactly the same lines as (\ref{carave2}), the only
difference here being that the asymptotic bias term cannot be given an
analytic expression but only bounded. It is then sufficient to
establish (%
\ref{carave2}), as we shall do below.

Following the notation introduced above, for each $L$ there exists
$\overline{%
\alpha}_{L}\in ( \alpha_{0}-\widehat{\alpha},\alpha_{0}+\widehat{%
\alpha} ) $ such that, with probability one:%
\[
( \widehat{\alpha}_{L}-\alpha_{0} ) =-\frac{S_{L}(\alpha_{0})}{%
Q_{L}(\overline{\alpha}_{L})},
\]
where $S_{L}(\alpha)$ is the score function corresponding to
$R_{L}(\alpha)
$, given by:
\[
S_{L}(\alpha)=\frac{\mrmdd}{\mrmdd\alpha}R(\alpha)=\frac{\widehat
{G}_{1}(\alpha)}{%
\widehat{G}(\alpha)}-
\frac{1}{\sum_{l=1}^{L}(2l+1)}\sum_{l=1}^{L}(2l+1)\log
l
\]
and%
\begin{eqnarray*}
Q_{L}(\alpha)&=&\frac{\mrmdd}{\mrmdd\alpha}S_{L}(\alpha)=
\frac{\mrmdd^{2}}{\mrmdd\alpha^{2}}%
R(\alpha)=\frac{\widehat{G}_{2}(\alpha)\widehat{G}(\alpha)-\widehat
{G}%
_{1}^{2}(\alpha)}{\widehat{G}{}^{2}(\alpha)}
\\[-2pt]
&=&\Biggl(\sum_{l=1}^{L}(2l+1)\bigl(\log^{2}l\bigr)\frac{\widehat
{C}_{l}}{l^{-\alpha}}%
\Biggl\{ \sum_{l=1}^{L}(2l+1)\frac{\widehat{C}_{l}}{l^{-\alpha}} \Biggr\}
- \Biggl\{ \sum_{l=1}^{L}(2l+1)(\log l)\frac{\widehat{C}_{l}}{l^{-\alpha
}}%
\Biggr\}^{2}\Biggr)\\[-2pt]
&&{}\Big/
\Biggl\{ \sum_{l=1}^{L}(2l+1)\frac{\widehat
{C}_{l}}{l^{-\alpha
}} \Biggr\}^{2},
\end{eqnarray*}
where $\widehat{G}(\alpha)$, $\widehat{G}_{1}(\alpha)$, $\widehat{G}%
_{2}(\alpha)$ are, respectively, the estimate of $G$ and its first and second
derivatives, as in Lemma \ref{lemmagiovedi}. By direct substitution, we have
immediately:\vspace*{-1pt}%
\[
S_{L}(\alpha)=\frac{1}{\sum_{l=1}^{L}(2l+1)}\sum_{l=1}^{L}(2l+1)
\log l \biggl\{ \frac{\widehat{C}_{l}}{\widehat{G}(\alpha)l^{-\alpha
}}-1 \biggr\}.\vspace*{-1pt}
\]
Now,
\begin{eqnarray*}
S_{L}(\alpha_{0}) &=&\frac{G_{0}}{\widehat{G}(\alpha_{0})}\frac{1}{%
\sum_{l=1}^{L}(2l+1)}
\sum_{l=1}^{L}(2l+1)\log l \biggl\{
\frac{\widehat
{C}_{l}%
}{G_{0}l^{-\alpha_{0}}}-\frac{\widehat{G}(\alpha_{0})}{G_{0}} \biggr\}
\\[-2pt]
&=&\frac{G_{0}}{\widehat{G}(\alpha_{0})}\overline{S}_{L}(\alpha_{0}),\vspace*{-1pt}
\end{eqnarray*}
where\vspace*{-1pt}%
\[
\overline{S}_{L}(\alpha_{0}) =\frac{1}{\sum_{l=1}^{L}(2l+1)}%
\sum_{l=1}^{L}(2l+1)\log l \biggl\{
\frac{\widehat{C}_{l}}{G_{0}l^{-\alpha
_{0}}}-1 \biggr\}
\]
and
\[
\frac{G_{0}}{\widehat{G}(\alpha_{0})} =1+\RMo_{p}(1)\qquad\mbox{as }%
L
\rightarrow\infty
\]
in view of Lemma \ref{lemmagiovedi}. Also%
\begin{eqnarray*}
\mathbb{E}\overline{S}_{L}(\alpha_{0}) &=&
\frac{1}{\sum_{l=1}^{L}(2l+1)}%
\sum_{l=1}^{L}(2l+1)
\log l \biggl\{ \frac{C_{l}}{G_{0}l^{-\alpha_{0}}}%
-1 \biggr\}
\\[-2pt]
&=&\frac{\kappa}{\sum_{l=1}^{L}(2l+1)}\sum_{l=1}^{L}(2l+1)
\frac{\log
l}{l}%
+\RMo\biggl(\frac{\log L}{L}\biggr)=\RMO\biggl(
\frac{\log L}{L}\biggr)\rightarrow0
\end{eqnarray*}
and%
%
\begin{equation}
\label{varSl} \lim_{L\rightarrow\infty}2L^{2}\operatorname{Var} \bigl\{
\overline{S}_{L}(\alpha_{0}) \bigr\} =1.
\end{equation}
In fact, we have:%
\[
\operatorname{Var} \bigl\{ \overline{S}_{L}(\alpha_{0}) \bigr\}
=V_{1}+V_{2}+V_{3},
\]
where%
\begin{eqnarray*}
V_{1} &=& \biggl\{ \frac{1}{\sum_{l=1}^{L}(2l+1)} \biggr\}^{2}\sum
_{l=1}^{L} ( 2l+1 )^{2} ( \log l
)^{2}\operatorname{Var} \biggl\{ \frac{\widehat{C}_{l}}{G_{0}l^{-\alpha_{0}}} \biggr\}
\\
&=& \biggl( \frac{1}{\sum_{l=1}^{L}(2l+1)} \biggr)^{2}2\sum
_{l=1}^{L} ( 2l+1 ) ( \log l )^{2};
\\
V_{2}&=& \biggl\{ \frac{1}{\sum_{l=1}^{L}(2l+1)} \biggr\}^{2} \Biggl(
\sum_{l=1}^{L} ( 2l+1 ) \log l
\Biggr)^{2}\operatorname{Var} \biggl( \frac
{\widehat{G%
}(\alpha_{0})}{G_{0}} \biggr);
\\
V_{3}&=&\frac{-2}{\sum_{l}(2l+1)}\sum_{l=1}^{L}
( 2l+1 ) \log l\operatorname{Cov} \biggl( \frac{\widehat{C}_{l}}{C_{l}},\frac{\widehat{G}(\alpha_{0})}{%
G_{0}} \biggr) \cdot
\frac{-2}{\sum_{l}(2l+1)}\sum_{l=1}^{L} ( 2l+1 )
\log l.
\end{eqnarray*}
Now because%
\begingroup
\abovedisplayskip=7pt
\belowdisplayskip=7pt
\begin{eqnarray}
\label{varG}
\operatorname{Var}
\biggl( \frac{\widehat{G}(\alpha_{0})}{G_{0}} \biggr) &=& \biggl\{ \frac
{1}{%
\sum_{l=1}^{L}(2l+1)} \biggr \}^{2}\sum_{l=1}^{L} ( 2l+1
)^{2}\operatorname{Var} \biggl\{
\frac{\widehat{C}_{l}}{G_{0}l^{-\alpha_{0}}} \biggr\} \nonumber
\nonumber\\[-9pt]\\[-9pt]
&=&\frac{2}{\sum_{l=1}^{L}(2l+1)};
\nonumber\\[-2pt]
\label{covarG}
\operatorname{Cov} \biggl( \frac{\widehat{C}_{l}}{C_{l}},\frac{\widehat{G}(\alpha_{0})}{G_{0}%
} \biggr) &=&
\frac{1}{\sum_{l^{\prime}=1}^{L}(2l^{\prime}+1)}%
\sum_{l^{\prime}=1}^{L}
\bigl( 2l^{\prime}+1 \bigr) \operatorname{Cov} \biggl( \frac{%
\widehat{C}_{l}}{C_{l}},
\frac{\widehat{C}_{l^{\prime}}}{C_{l^{\prime
}}}%
\biggr)
\nonumber\\[-9pt]\\[-9pt]
&=&\frac{2}{\sum_{l^{\prime}=1}^{L}(2l+1)};\nonumber
\end{eqnarray}
we have
\begin{eqnarray*}
\operatorname{Var} \bigl\{ \overline{S}_{L}(\alpha_{0}) \bigr\}
&=&\frac{2}{ ( \sum_{l=1}^{L}(2l+1) )^{3}} \Biggl( \sum_{l=1}^{L}
( 2l+1 ) \sum_{l=1}^{L} ( 2l+1 ) ( \log l
)^{2}- \Biggl( \sum_{l=1}^{L} (
2l+1 ) \log l \Biggr)^{2} \Biggr)
\\[-2pt]
&=&\frac{2}{L^{6}}\frac{L^{4}}{4}=\frac{1}{2L^{2}}
\end{eqnarray*}
by using (\ref{sumLcorol}) and (\ref{slog0}) with $s=0$ to obtain (\ref%
{varSl}). In order to establish the central limit theorem, it is sufficient
to perform a careful analysis of fourth-order cumulants (note our statistics
belong to the second-order Wiener chaos with respect to a Gaussian white
noise random measure). Write:%
\[
LS_{L}(\alpha_{0})=\frac{1}{L+\RMO_{L}(1)}\sum
_{l}(A_{l}+B_{l}),
\]
where%
%
\begin{eqnarray}
\label{Al} A_{l} &=&(2l+1)\log l \biggl\{ \frac{\widehat{C}_{l}}{C_{l}}-1 \biggr
\},
\\[-2pt]
\label{Bl} B_{l} &=&(2l+1)\log l \biggl\{ \frac{\widehat{G}_{L}(\alpha_{0})}{G_{0}}%
-1 \biggr\}.
\end{eqnarray}
In the \hyperref[app]{Appendix}, we show that%
\[
\frac{1}{L^{4}}\operatorname{cum} \biggl\{ \sum_{l_{1}}(A_{l_{1}}+B_{l_{1}}),
\sum_{l_{2}}(A_{l_{2}}+B_{l_{2}}), \sum
_{l_{3}}(A_{l_{3}}+B_{l_{3}}),\sum
_{l_{4}}(A_{l_{4}}+B_{l_{4}}) \biggr\}
=\RMO_{L}\biggl(\frac{\log^{4}L}{L^{2}}\biggr),
\]
whence the central limit theorem follows easily from results in
\cite{nourdinpeccati}. Indeed, using recent results from the latter
authors a
stronger result follows, that is,
\[
d_{\mathrm{TV}}\Biggl(\sum_{l=1}^{L}X_{l;L},Z
\Biggr)=\RMO\biggl(\frac{1}{L}\biggr),\qquad Z\stackrel{d} {=}%
\mathcal{N}(0,1),\vadjust{\goodbreak}
\]
\endgroup
where $d_{\mathrm{TV}}(W,V)$ denotes the total variation distance between the random
variables $W,V$, that is,
\[
d_{\mathrm{TV}}(W,V)=\sup_{x}\bigl\llvert \Pr \{ W\in B \} -\Pr \{
V\in B \} \bigr\rrvert \qquad\mbox{any Borel set }B.
\]
Also%
\[
\frac{L}{\log L}\mathbb{E}\overline{S}_{L}(\alpha_{0})=
\kappa\frac{L}{
\sum_{l=1}^{L}(2l+1)}\sum_{l=1}^{L}
\frac{(2l+1)}{l}\frac{\log l}{\log L} 
+\RMo(1)\rightarrow-\kappa\qquad
\mbox{as }L\rightarrow\infty.
\]
Let us now focus on the second order derivative. From consistency, it is
sufficient to focus on $\llvert \alpha-\alpha_{0}\rrvert <2$; here
we can apply again Lemma \ref{lemmagiovedi}, replacing the random quantities
$\widehat{G}_{k}(\alpha)$ with the corresponding deterministic $%
G_{k}(\alpha)$ values, to obtain
\[
Q_{L}(\alpha)=\frac{G_{2}(\alpha)G(\alpha)-G_{1}^{2}(\alpha)}{%
G^{2}(\alpha)}+\RMo_{p}(1),
\]
uniformly over $\alpha$. It is convenient to write%
\[
\frac{G_{2}(\alpha)G(\alpha)-G_{1}^{2}(\alpha)}{G^{2}(\alpha)}=\frac
{%
Q_{L}^{\mathrm{num}}(\alpha)}{Q_{L}^{\mathrm{den}}(\alpha)}.
\]
Let us start by studying $Q_{L}^{\mathrm{den}}(\alpha)$. We have, by using (\ref
{slog0}) with $s=0$ and $s=\alpha-\alpha_{0}$:
\begin{eqnarray*}
\frac{Q_{L}^{\mathrm{den}}(\alpha)}{L^{2 ( \alpha-\alpha_{0} ) }} &=& 
\frac{1}{L^{2 ( \alpha-\alpha_{0} ) }} \Biggl( \frac{1}{%
\sum_{l=1}^{L}(2l+1)}
\sum_{l=1}^{L}(2l+1)\frac{G_{0}l^{-\alpha_{0}}}{%
l^{-\alpha}}
\Biggr)^{2}
\\[-2pt]
&=&G_{0}^{2} \biggl( \frac{1}{ ( 1+({\alpha-\alpha_{0}})/{2} )^{2}}+
\RMo_{L}(1) \biggr).
\end{eqnarray*}
Consider now $Q_{L}^{\mathrm{num}}(\alpha)$, where we have:%
\begin{eqnarray*}
&& \frac{Q_{L}^{\mathrm{num}}(\alpha)}{L^{2 ( \alpha-\alpha_{0} ) }}
\\
&&\quad= \biggl( \frac{G_{0}L^{- ( \alpha-\alpha_{0} ) }}{%
\sum_{l=1}^{L} ( 2l+1 ) } \biggr)^{2}\\
&&\qquad{}\times \Biggl[ \Biggl( \sum
_{l=1}^{L} ( 2l+1 ) \frac{l^{-\alpha_{0}}}{l^{-\alpha
}}
\log^{2}l \Biggr) \Biggl( \sum_{l=1}^{L}(2l+1)
\frac{l^{-\alpha
_{0}}}{l^{-\alpha}}%
\Biggr) - \Biggl( \sum_{l=1}^{L}(2l+1)
\frac{l^{-\alpha_{0}}}{l^{-\alpha
}}%
\log l \Biggr)^{2} \Biggr]
\\
&&\quad=\frac{G_{0}^{2}}{L^{4+2 ( \alpha-\alpha_{0} ) }}\! \Biggl[\! \Biggl( \sum_{l=1}^{L}
( 2l+1 ) \frac{l^{-\alpha_{0}}}{l^{-\alpha
}}\log^{2}l \Biggr)\! \Biggl( \sum
_{l=1}^{L}(2l+1)\frac{l^{-\alpha
_{0}}}{l^{-\alpha}}%
\Biggr)\,{-}\,\Biggl( \sum_{l=1}^{L}(2l+1)
\frac{l^{-\alpha_{0}}}{l^{-\alpha
}}%
\log l \Biggr)^{\!\!2} \Biggr]
\\
&&\quad=G_{0}^{2} \biggl[ \frac{1}{4 ( 1+({\alpha-\alpha_{0}})/{2} )^{4}%
} \biggr] +
\RMo_{L} ( 1 )
\end{eqnarray*}
by using (\ref{sumLcorol}), $s=\alpha-\alpha_{0}$. Combining all
terms, we find that, uniformly over $\alpha$%
\begingroup
\abovedisplayskip=7pt
\belowdisplayskip=7pt
\[
Q_{L}(\alpha)=\frac{G_{0}^{2}({1}/({4 ( 1+({\alpha-\alpha_{0}})/{2%
} )^{4}}))+\RMo_{L}(1)}{G_{0}^{2} ( {1}/{ ( 1+({\alpha
-\alpha_{0}})/{2} )^{2}}+\RMo_{L}(1) ) }=\frac{1}{4 ( 1+
({%
\alpha-\alpha_{0}})/{2} )^{2}}+
\RMo_{L}(1).
\]
Finally, from the consistency result%
\[
\biggl( 1+\frac{\overline{\alpha}_{L}-\alpha_{0}}{2} \biggr)^{2}\stackrel{%
\mathbb{P}} {\longrightarrow}1,\qquad Q_{L}(\overline{\alpha}_{L})
\stackrel{\mathbb{P}} {\longrightarrow}\frac{1}{4}
\]
and thus, as claimed:%
\[
\frac{\sqrt{2}L}{4}\frac{S_{L}(\alpha_{0})}{Q_{L}(\overline{\alpha
}_{L})}%
\stackrel{d} {\longrightarrow}
\mathcal{N} ( -\sqrt{2}\kappa,1 ).
\]
\upqed
\end{pf}

In the \hyperref[app]{Appendix} we describe in details the results concerning the analysis
of fourth-order cumulants.\vspace*{-2pt}
%
\begin{remark}
\label{corr} In the statement of the previous theorem, we decided to report
normalization factors in the neatest possible form. A careful
inspection of
the proofs reveals however that the asymptotic result in (\ref
{carave2}) and
(\ref{carave3}) can be improved in finite samples introducing a correction
factor $c_{L}=\frac{1}{L}\sum_{l=1}^{L}\frac{\log l}{\log L}\rightarrow1$,
as $L\rightarrow\infty$, as follows%
\[
\frac{L}{4\log L\times c_{L}} ( \widehat{\alpha}_{L}-\alpha_{0} )
\longrightarrow_{p}\kappa
\]
under Condition \ref{A2}, and
\[
\frac{\sqrt{2}L}{4\times c_{L}} ( \widehat{\alpha}_{L}-\alpha_{0} )
\stackrel{d} {\longrightarrow}\mathcal{N} ( 0,1 ),
\]
under Condition \ref{A3}. Note that $c_{L}<1$ for all finite $L$, whence
the asymptotic bias and variance are slightly underestimated in Theorem
\ref%
{clt0}. For instance, the correction factors for $L=1000,2000,4000$ are,
respectively, $c_{1000}\simeq0.86$, $c_{2000}\simeq0.87$, and $%
c_{4000}\simeq0.88$.\vspace*{-2pt}
\end{remark}
%
\begin{remark}
Under Condition \ref{A2}, it is possible to implement consistent estimates
for the parameter~$\kappa$, with a slower rate of convergence. We leave
this issue as a topic for further research.\vspace*{-2pt}
\end{remark}

The previous result provides a sharp rate of convergence for the spherical
Whittle estimator. However in the general case the asymptotic bias term
$-\sqrt{2}\kappa$ is unknown, which makes inference unfeasible. To
address these issues, we shall consider in the next section an
alternative \textit{narrow-band} estimator (compare \cite{Robinson})
which achieves an unbiased limiting distribution, to the price of a log
factor in the rate of convergence.\vspace*{-3pt}
\endgroup

\section{Narrow-band estimates} \label{narrowband}\vspace*{-3pt}

In the previous section, we have shown that under Conditions \ref{A1},
\ref{A2}, it is possible to establish a rate of convergence for the
spherical\vadjust{\goodbreak}
Whittle estimates; however, due to the presence of an asymptotic bias term,
statistical inference turned out to be unfeasible. The purpose of this
section is to propose a narrow band estimator allowing for feasible
inference under broad circumstances. We start from the following definition.
%
\begin{definition}
The \textup{Narrow--Band Spherical Whittle estimator} for the parameters $%
\vartheta=(\alpha,G)$ is provided by%
\[
(\widehat{\alpha}_{L;L_{1}},\widehat{G}_{L;L_{1}}):=\arg
\min_{\alpha
,G}\sum_{l=L_{1}}^{L} \biggl\{
(2l+1)\frac{\widehat{C}_{l}}{Gl^{-\alpha}} 
-(2l+1)\log\frac{\widehat{C}_{l}}{Gl^{-\alpha}} \biggr\}
\]
or equivalently%
%
\begin{eqnarray}
\label{narrowest} \widehat{\alpha}_{L;L_{1}} &=&\arg\min_{\alpha}R_{L;L_{1}}
\bigl(\alpha, \widehat{G}(\alpha)\bigr),
\nonumber\\[-8pt]\\[-8pt]
R_{L;L_{1}}\bigl(\alpha,\widehat{G} ( \alpha ) \bigr) &=& \Biggl( \log
\widehat{G}_{L;L_{1}}(\alpha)-\frac{\alpha}{\sum_{l=L_{1}}^{L}(2l+1)}%
\sum
_{l=L_{1}}^{L}(2l+1)\log l \Biggr),
\nonumber
\end{eqnarray}
where $L_{1}<L$ is chosen such that%
\[
L-L_{1}\rightarrow\infty,\qquad \frac{L}{L_{1}}=1+\RMO\biggl(
\frac{1}{\log L}\biggr)%
\qquad\mbox{as }L\rightarrow\infty.
\]
\end{definition}

We can write%
\[
L_{1}=L \bigl( 1-g ( L ) \bigr),
\]
where
\[
g ( L ) =g (
L;L_{1} ) =1-\frac{L_{1}}{L}=\RMO\biggl(\frac{1}{\log L}\biggr),\qquad
\lim_{L\rightarrow\infty}\bigl(L\times g(L)\bigr)=\infty.
\]
%
\begin{theorem}
\label{theonarrowband}Let $\widehat{\alpha}_{L;L_{1}}$ defined as in
(\ref%
{narrowest}). Then under Condition \ref{A2} we have
\[
\frac{L\cdot\sqrt{g^{3} ( L ) }}{\sqrt{12}} ( \widehat {\alpha}%
_{L;L_{1}}-
\alpha_{0} ) \stackrel{d} {\longrightarrow}\mathcal {N} ( 0,1 ).
\]
\end{theorem}
\begin{pf}
The proof of the consistency for $\widehat{\alpha}_{L;L_{1}}$ can be
carried out analogously to the argument provided in Section \ref%
{asconsistency}, and hence is omitted for brevity's sake. The proof for the
central limit theorem can also be carried along the same lines as done
earlier, noting in particular that for the form (\ref{reg2}) of $C_{l}$
under Condition \ref{A2}%
\begin{eqnarray*}
\mathbb{E}\overline{S}_{L;L_{1}}(\alpha_{0}) &=&
\frac{1}{%
\sum_{l=L_{1}}^{L}(2l+1)}\sum_{l=L_{1}}^{L}(2l+1) \{
\log l \} \biggl\{ \frac{C_{l}}{G_{0}l^{-\alpha_{0}}}-\frac{\widehat
{G}_{L;L_{1}}}{%
G_{0}} \biggr\}
\\
&=&\frac{\kappa}{\sum_{l=L_{1}}^{L}(2l+1)}\sum_{l=L_{1}}^{L}
\biggl[ (2l+1)%
\frac{\log l}{l}-\frac{\sum_{l=L_{1}}^{L}(2+{1}/{l})}{%
\sum_{l=L_{1}}^{L}(2l+1)} \biggr]
\\
&=&\kappa\frac{\log L_{1}}{L_{1}}+\RMo \biggl( \frac{\log
L_{1}}{L_{1}} \biggr)\\
&=&\RMO
\biggl(\frac{\log L_{1}}{L_{1}}\biggr)
\end{eqnarray*}
and
\begin{eqnarray*}
L\cdot\sqrt{g^{3} ( L ) }\mathbb{E} \bigl[ \overline{S}%
_{L;L_{1}}(
\alpha_{0}) \bigr] &=&\RMO \biggl( \frac{\log L_{1}}{L_{1}} \biggr)
\frac{L}{\log^{{3}/{2}}L}\\
&=&\RMO\biggl(\frac{L}{L_{1}}\biggr)\RMO \biggl( \frac{\log L}{\log^{{3}/{2}}L}
\biggr)\\
&=&\RMO \biggl( \frac{1}{\log^{{1}/{2}}L} \biggr) =\RMo_{L}(1).
\end{eqnarray*}
On the other hand%
%
\begin{eqnarray}
\label{narrowvar} && \operatorname{Var} \bigl\{ \overline{S}_{L;L_{1}}(
\alpha_{0}) \bigr\}
\nonumber
\\
&&\quad=\frac{1}{ [ \sum_{l=L_{1}}^{L}(2l+1) ]^{2}}\operatorname{Var} \Biggl\{ \sum_{l=L_{1}}^{L}(2l+1)
\{ \log l \} \biggl( \frac{\widehat
{C}_{l}}{%
G_{0}l^{-\alpha_{0}}}-\frac{\widehat{G}_{L;L_{1}} ( \alpha )
}{%
G_{0}} \biggr) \Biggr\}
\nonumber\\[-8pt]\\[-8pt]
&&\quad=\frac{2}{ [ \sum_{l=L_{1}}^{L}(2l+1) ]^{3}} \Biggl( \sum_{l=L_{1}}^{L}(2l+1)
\sum_{l=L_{1}}^{L}(2l+1) \bigl\{
\log^{2}l \bigr\} - \Biggl( \sum_{l=L_{1}}^{L}(2l+1)
\{ \log l \} \Biggr)^{2} \Biggr)\qquad
\nonumber\\
&&\quad=\frac{2}{ [ \sum_{l=L_{1}}^{L}(2l+1) ]^{3}}Z_{L;g (
L ) } ( 0 )
\nonumber
\end{eqnarray}
by using (\ref{varG}) and (\ref{covarG}) and following the notation of
Proposition \ref{sumLintegralscomplete} with $s=0$.

Proposition \ref{sumLintegralscomplete} leads to:
\[
\tfrac{1}{4}Z_{L;g ( L ) }=\tfrac{1}{3}g^{4} ( L )
L^{4}+\RMo \bigl( g^{4} ( L ) L^{4} \bigr),
\]
while
\[
\Biggl[ \sum_{l=L_{1}}^{L}(2l+1)
\Biggr]^{3}= \bigl( L^{2}-L_{1}^{2}
\bigr)^{3}=8L^{6}g^{3} ( L ) +\RMo_{L}
\bigl( L^{6}g^{3} ( L ) \bigr).
\]
By substituting these results in (\ref{narrowvar}), we obtain
\[
\operatorname{Var} \bigl\{ \overline{S}_{L;L_{1}}(\alpha_{0}) \bigr\} =
\frac{g (
L ) }{12L^{2}}=\frac{1}{12L^{2}\log ( L ) }.
\]
Rewrite now the term $Q_{L_{1}L}(\alpha)$ as%
\[
Q_{L_{1};L}(\alpha)=\frac{Q_{\dot{L}_{1};L}^{\mathrm{num}}(\alpha)}{Q_{\dot{L}%
_{1};L}^{\mathrm{den}}(\alpha)},
\]
where we have:%
\begin{eqnarray*}
Q_{L_{1};L}^{\mathrm{num}}(\alpha)&=&\frac{G_{0}^{2}}{ ( \sum_{l=1}^{L} (
2l+1 )  )^{2}}Z_{L,g ( L ) } (
s ),
\\
Q_{L_{1};L}^{\mathrm{den}}(\alpha)&=&G_{0}^{2} \biggl(
\frac{1}{\sum_{l=L_{1}}^{L}(2l+1)}%
\biggr)^{2} \Biggl( \sum
_{l=L_{1}}^{L}(2l+1)l^{s}
\Biggr)^{2},
\end{eqnarray*}
where $s=\alpha-\alpha_{0}$.

From (\ref{slog0}) and (\ref{venerdi}), we have%
\begin{eqnarray*}
Q_{L_{1};L}^{\mathrm{den}}(\alpha) &=&\frac{G_{0}^{2}}{ ( 1+
{s}/{2} )^{2}}
\frac{L^{4 ( 1+{s}/{2} ) } ( 1- ( 1-g (
L )  )^{2 ( 1+{s}/{2} ) } )^{2}}{L^{4} (
1- ( 1-g ( L )  )^{2} )^{2}}+\RMo_{L} ( 1 )
\\
&=&\frac{4G_{0}^{2}L^{2s}g^{2} ( L ) }{ ( 1- ( 1-g (
L )  )^{2} )^{2}}+\RMo_{L} ( 1 ).
\end{eqnarray*}

Consider now $Q_{L_{1};L}^{\mathrm{num}}(\alpha)$, where we have:%
\[
Q_{L_{1};L}^{\mathrm{num}}(\alpha)=G_{0}^{2}
\frac{L^{2s}g^{4} ( L )
K (
s ) }{ ( 1- ( 1-g ( L )  )^{2} )^{2}}%
+\RMo_{L} ( 1 ).
\]
Combining the two results, we obtain:%
\[
\lim_{L\rightarrow\infty}Q_{L_{1};L}(\alpha)=\frac{g^{2} (
L )
K ( s ) }{4}.
\]
Finally, from the consistency results, we have:%
\[
\frac{12}{g^{2} ( L ) }Q_{L_{1};L}(\overline{\alpha })\rightarrow_{p}1.
\]

The analysis of fourth-order moments is exactly the same as in the previous
section, and the result follows accordingly.
\end{pf}
%
\begin{remark}
It should be noted that an asymptotic unbiased estimator is obtained with
the loss of only a logarithmic term to the power $3/2$ in the rate of
convergence. This result highlights the fact that for spherical random
fields the highest order multipoles have a dominating role in the estimation
procedure. This is a consequence of the peculiar features of Fourier
analysis under isotropy -- the number of random spherical harmonic
coefficients grows linearly with the order of the multipoles.
\end{remark}
%
\begin{remark}
A careful inspection of the proof reveals that, in case it is assumed that
the scale parameter $G=G_{0}$ is known, a faster rate of convergence
results. This is consistent with results from \cite{wu}, where stationary
Gaussian processes on $\mathbb{R}^{d}$ are considered and asymptotic
Gaussianity for the spectral index and the scale parameters are separately
established.
\end{remark}

\section{Estimation with noise} \label{noise}

The previous sections have been developed under an overly simplified
assumption, that is, the condition that the random spherical harmonic
coefficients $ \{ a_{lm} \} $ can be observed without noise. Of
course, this assumption is untenable under realistic experimental
circumstances. The purpose of the present section is to show how our
approach can be extended to cope with noise. More precisely, and following
earlier work by \cite{polenta,glm} (see also \cite{marpecbook}), we
shall assume that observations the observed spherical field takes the form
\[
\RMO ( x ):=T ( x ) +N ( x ),\qquad x\in\mathbb{%
S}^{2},
\]
where $N ( x ) $ is taken to be a zero-mean, square-integrable,
isotropic random field representing noise, which is Gaussian and independent
from the signal $T(x)$. The spherical harmonic coefficients then become%
\[
a_{lm}=\int_{\mathbb{S}^{2}}\RMO ( x ) \overline{Y}_{lm}
( x ) \mrmd x=a_{lm}^{T}+a_{lm}^{N},
\]
where the set $ \{ a_{lm}^{T},a_{lm}^{N} \} $ are associated,
respectively, to the random field $T ( x ),N(x)$. More precisely
%
\begin{condition}
\label{gammacond}The random field $N ( x ) $ is Gaussian and
isotropic, independent form $T ( x ) $ and with angular power
spectrum
\[
C_{N,l}=G_{N}l^{-\gamma},\qquad \gamma>2,
G_{N}>0.
\]
\end{condition}

Clearly%
\[
\widehat{C}_{l}=\frac{1}{2l+1} \Biggl[ \sum
_{m=-l}^{l}\bigl\llvert a_{lm}^{T}
\bigr\rrvert^{2}+\sum_{m=-l}^{l}
\bigl\llvert a_{lm}^{N}\bigr\rrvert^{2}+2\Re
\Biggl( \sum_{m=-l}^{l}a_{lm}^{T}
\overline{a}_{lm}^{N} \Biggr) %
\Biggr],
\]
so that%
\[
\mathbb{E} ( \widehat{C}_{l} ) =C_{T,l}+C_{N,l},\qquad
\operatorname{Var} ( \widehat{C}_{l} ) =\frac{2}{2l+1} \bigl(
C_{T,l}^{2}+C_{N,l}^{2} \bigr).
\]
The naive estimator $ \{ \widehat{C}_{l} \} $ is then biased for
the power spectrum of interest $ \{ C_{T,l} \} $. In the
cosmological literature, this issue is addressed by two alternative methods:
\begin{itemize}
\item(A) For most experimental set-ups, it may be reasonable to assume that
the angular power spectrum is known a priori, and hence can be
subtracted from the data. This leads to the so-called \textit{auto-power
spectrum} estimator.

\item(B) Most experiments in a CMB framework are actually
\textit{multi-channel}, that is, they provide a vector of observations,
such that the signal $a_{lm}^{T}$ is constant across all components,
while noise is independent from one component to the other. This leads
easily to an unbiased estimator even without the assumption that the
noise angular power spectrum is known in advance -- this estimator is
known as the \textit{cross-power spectrum.}
\end{itemize}

A detailed comparison among the two estimators and consistent tests on
the functional form of the noise power spectrum are again discussed in
\cite{polenta,glm} (see also \cite{marpecbook}, Chapter 8.3). Here, for
brevity and notational simplicity we shall focus on case (A), that is,
on the unbiased estimator:
\begin{eqnarray*}
\widetilde{C}_{l} &=&\frac{1}{2l+1}\sum
_{m=-l}^{l} \bigl[ \bigl\llvert a_{lm}^{T}+a_{lm}^{N}
\bigr\rrvert^{2} \bigr] -C_{l}^{N}
\\
&=&\frac{1}{2l+1}\sum_{m=-l}^{l} \Biggl[
\bigl\llvert a_{lm}^{T}\bigr\rrvert^{2}+\bigl
\llvert a_{lm}^{N}\bigr\rrvert^{2}+2\Re \Biggl(
\sum_{m=-l}^{l}a_{lm}^{T}
\overline{a}_{lm}^{N} \Biggr) \Biggr] -C_{N,l},
\end{eqnarray*}
where $\mathbb{E} ( \widetilde{C}_{l} ) =C_{T,l}$ and
\begin{eqnarray*}
\operatorname{Var} \biggl( \frac{\widetilde{C}_{l}}{C_{T,l}} \biggr) &=&\frac
{2}{2l+1} \biggl(
1+%
\frac{C_{N,l}^{2}}{C_{T,l}^{2}}+2\frac{C_{N,l}}{C_{T,l}} \biggr)
\\
&=&\frac{2}{2l+1} \biggl( \biggl( 1+ \biggl( \frac{G_{N}}{G_{0}} \biggr)
l^{- ( \gamma-\alpha_{0} ) } \biggr)^{2}+\RMO \bigl( l^{-\min
(
2 ( \gamma-\alpha_{0} ), ( \gamma-\alpha_{0} )
) } \bigr)
\biggr).
\end{eqnarray*}
%
\begin{remark}
There are three asymptotic regimes for the behaviour of $\operatorname{Var} (
\widetilde{C}_{l}/C_{T,l} )$:
\end{remark}
\begin{enumerate}
\item$\alpha_{0}<\gamma$, where
\[
\operatorname{Var} \biggl( \frac{\widetilde{C}_{l}}{C_{T,l}} \biggr) =\frac{2}{2l+1} \bigl( 1+\RMO
\bigl( l^{- ( \gamma-\alpha_{0} ) } \bigr) \bigr).
\]

\item$\alpha_{0}=\gamma$, where
\[
\operatorname{Var} \biggl( \frac{\widetilde{C}_{l}}{C_{T,l}} \biggr) =\frac{2}{2l+1} \biggl( \biggl( 1+
\frac{G_{N}}{G_{0}} \biggr)^{2}+\RMO \bigl( l^{-1} \bigr)
\biggr).
\]

\item$\alpha_{0}>\gamma$, so that
\[
\operatorname{Var} \biggl( \frac{\widetilde{C}_{l}}{C_{T,l}} \biggr) =\frac{2}{2l+1} \bigl(
G_{N}^{2}l^{-2 ( \gamma-\alpha_{0} ) }+\RMO \bigl( l^{-\min (
2\alpha_{0}, ( \gamma+\alpha_{0} )  ) }
\bigr) \bigr).
\]
\end{enumerate}

In the first case the presence of instrumental noise is asymptotically
negligible and the results of the previous sections will remain unaltered.
As before, we define:%
%
\begin{eqnarray}
\label{noiseest} \widetilde{G}_{L}&=&\frac{1}{\sum_{l=1}^{L}(2l+1)}\sum
_{l=1}^{L}(2l+1)\frac{%
\widetilde{C}_{l}}{l^{-\alpha}};
\nonumber\\[-8pt]\\[-8pt]
\widetilde{\alpha}_{L}&=&\arg\min_{\alpha>2}R_{L}^{\mathrm{noise}}(
\alpha ),
\nonumber
\end{eqnarray}
where%
\[
R_{L}^{\mathrm{noise}}(\alpha)=\log\frac{1}{\sum_{l=1}^{L}(2l+1)}%
\sum
_{l=1}^{L}(2l+1)\frac{\widetilde{C}_{l}}{l^{-\alpha}}-
\frac{\alpha
}{%
\sum_{l=1}^{L}(2l+1)}\sum_{l=1}^{L}(2l+1)\log
l.
\]

The proof of the consistency of the estimator $\widetilde{\alpha}_{L}$
follows strictly the argument that was provided above in the noiseless case.
Indeed, for $\alpha_{0}<\gamma$ noise is asymptotically negligible, and
all proofs are basically unaltered; for $\alpha_{0}\geq\gamma+1$
consistency can no longer be established. Finally, for $\gamma<\alpha_{0}<\gamma+1$ the arguments go through with some changes in the
convergence rates; details are provided in the \hyperref[app]{Appendix}.
%
\begin{theorem}
\label{theonoise}Let $\widetilde{\alpha}_{L}$ defined as in (\ref
{noiseest}%
). Then under Conditions \ref{A2} and \ref{gammacond}, we have for
$\gamma
>\alpha_{0}-1$
\[
\frac{L}{4\log L} ( \widetilde{\alpha}_{L}-\alpha_{0} )
\longrightarrow_{p}-\kappa.
\]
If moreover Condition \ref{A3} holds, we have that,
\begin{eqnarray*}
\frac{\sqrt{2}L}{4} ( \widetilde{\alpha}_{L}-\alpha_{0} )
&\stackrel{d} {\longrightarrow}&\mathcal{N} ( 0,1 ) \qquad\mbox{for }
\alpha_{0}<\gamma;
\\
\frac{\sqrt{2}L}{4} \biggl( 1+\frac{G_{N}}{G_{0}} \biggr)^{2} (
\widetilde{%
\alpha}_{L}-\alpha_{0} ) &
\stackrel{d} {\longrightarrow}&\mathcal {N}%
( 0,1 ) \qquad\mbox{for }
\alpha_{0}=\gamma;
\\
L^{1- ( \alpha_{0}-\gamma ) }\frac{\sqrt{2}}{4\sqrt{H (
\alpha_{0}-\gamma ) }} \biggl( \frac{G_{0}}{G_{N}} \biggr) (
\widetilde{\alpha}_{L}-\alpha_{0} ) &\stackrel{d} {
\longrightarrow}& \mathcal{N} ( 0,1 ) \qquad\mbox{for }\gamma<
\alpha_{0}<\gamma+1,
\end{eqnarray*}
where%
\[
H ( u ):= \biggl( \frac{7+4u+u^{2}}{4 ( 1+u )^{3}} \biggr).
\]
\end{theorem}

The rate of convergence and the asymptotic variance of $\widetilde
{\alpha}%
_{L}$, for example, $L^{1- ( \alpha_{0}-\gamma ) }$ depend on the unknown
parameters $\alpha_{0},G_{0}$. However, these unknown values can be
replaced by their consistent estimates, with no effect on the asymptotic
results; indeed it is easily seen that, for instance,
\[
L^{1- ( \widetilde{\alpha}_{L}-\gamma ) }\frac{\sqrt {2}}{4\sqrt{%
H ( \widetilde{\alpha}_{L}-\gamma ) }} \biggl( \frac{\widetilde
{G}%
_{0}}{G_{N}} \biggr) (
\widetilde{\alpha}_{L}-\alpha_{0} ) \stackrel{d} {
\longrightarrow}\mathcal{N} ( 0,1 ) \qquad\mbox{for }%
\gamma<
\alpha_{0}<\gamma+1,
\]
because $ ( \widetilde{\alpha}_{L}-\alpha_{0} ) =\RMo_{p}(\log L)$,
whence the result follows by noting that%
\[
\frac{G_{N}}{\widetilde{G}_{L} ( \alpha_{0} ) }\rightarrow_{p}%
\frac{G_{N}}{G_{0}},\qquad
\frac{L^{1- ( \widetilde{\alpha}%
_{L}-\gamma ) }\sqrt{H ( \alpha_{0}-\gamma )
}}{L^{1- (
\alpha_{0}-\gamma ) }\sqrt{H ( \widetilde{\alpha}_{L}-\gamma
) }}\rightarrow_{p}1\qquad\mbox{as }L\rightarrow\infty.
\]
Analogous extensions to address observational noise can be considered for
the narrow-band estimators; this case is omitted, however, for brevity's
sake.

\section{Numerical results} \label{numerical}

In this section, we present some numerical evidence to support the asymptotic
results provided earlier. More precisely, using the statistical
software R,
for given fixed values of $L$, $\alpha_{0}$ and $G_{0}$ and the alternative
conditions discussed in the previous section, we sample random values for
the angular power spectra $\widehat{C}_{l}$ and we implement standard and
narrow-band estimates. We start by analyzing the simplest model, that
is, the
one corresponding to Condition \ref{A3}. Here we fixed $G_{0}=2$. In Figure
\ref{figur1}, we report the distribution of $\widehat{\alpha}_{L}-\alpha_{0}$
normalized by a factor $\sqrt{2}L/4$. In Table \ref{tabl1}, we report instead the
sample frequencies corresponding to the quantiles $%
q=0.05,0.25,0.50.0.75,0.95 $ for a $\mathcal{N} ( 0,1 ) $
distribution.

%
\begin{figure}

\includegraphics{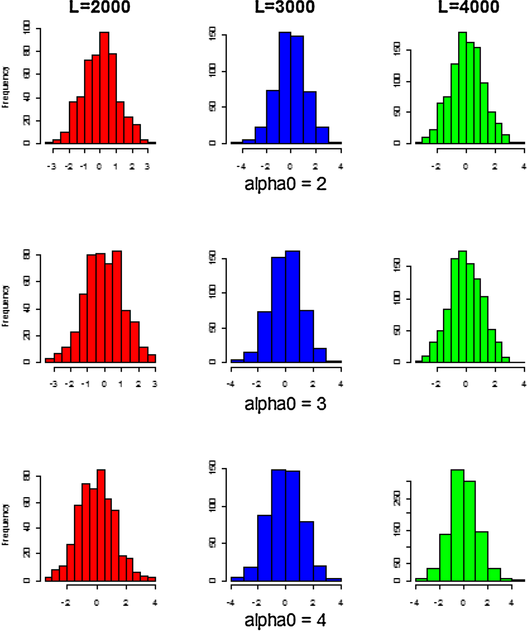}

\caption{Distribution of normalized $ ( \widehat{\alpha
}_{L}-\alpha_{0} ) $ by varying $L$ and $\alpha_{0}$, under
Condition \protect\ref{A3}.}\label{figur1}
\end{figure}

Table \ref{tabl2} provides the results for the classical Shapiro--Wilk
Gaussianity test performed on simulations obtained by varying
$\alpha_{0}$ and the number of multipoles $L$. Asymptotic Gaussianity is
clearly supported.

Let us now focus on the more general Condition \ref{A2}. Figure \ref{figur2} represents
the empirical distribution of $ ( L/4\sqrt{2}\log L )  (
\widehat{\alpha}_{L}-\alpha_{0} ) $ in case $\alpha_{0}=3$,
$\kappa
=1$ and the corresponding narrow-band estimates, whose results are
summarized in Table \ref{tabl3}. The improvement in the bias factor with the latter
procedure is immediately evident.

Once more, asymptotic Gaussianity is strongly supported by the Shapiro--Wilk
test, see again Table \ref{tabl3}.

Considering the correction term $c_{L}$ from Remark \ref{corr}, the sample
bias is consistent with the asymptotic value to three decimal digits.

In Figure \ref{figur3}, we report the results obtained on a set of simulations under
Condition \ref{A1}, where we have:
\[
G(l)=G_{0} \biggl\{ 1+\frac{1}{l}-\frac{1}{l^{2}} \biggr\}
\]
with $G_{0}=2$, $\alpha_{0}=4$, $L=4000$, $L_{1}=3750$.

\begin{table}
\tablewidth=280pt
\caption{Quantiles of $L/4\sqrt{2}\log L ( \widehat{\alpha
}_{L}-\alpha_{0} ) $}\label{tabl1}
\begin{tabular*}{\tablewidth}{@{\extracolsep{\fill}}lllllllll@{}}
\hline
& & \multicolumn{7}{l@{}}{Sample frequencies} \\[-4pt]
& & \multicolumn{7}{l@{}}{\hrulefill} \\
\multicolumn{1}{@{}l}{$\alpha_{0}$} & \multicolumn{1}{l}{$L$}
& $-1.96$ & $-1$ & $-0.68$ & $0$ & $0.68$
& $1 $ & $1.96$ \\
\hline
2 & 2000
& 4 & 19.2 & 29.2 & 48.6 & 22.8 & 14.2 & 4 \\
& 3000
& 4.5 & 18.4 & 26.8 & 51 & 23.33 & 14.36 & 3.6 \\
& 4000 & 4.4 & 17.7 & 25.2 & 49.1 & 23.43 & 13.87 & 4.8 \\
[4pt]
3 & 2000
& 4.4 & 19.2 & 29.2 & 51.5 & 24.03 & 15.17 & 3.6 \\
& 3000 & 4.3 & 18.4 & 26.8 & 48.9 & 23.2 & 13.43
& 3.8
\\
& 4000
& 4.2 & 17.9 & 26.4 & 50.8 & 23.07 & 14.13 & 3.7 \\
[4pt]
4 & 2000
& 4.4 & 21.6 & 30.2 & 50.9 & 22.73 & 14.94 & 5.5 \\
& 3000 & 4.2 & 21.2 & 29.8 & 50.4 & 25.07 & 15.87 & 4.3 \\
& 4000
& 4.2 & 17.9 & 27.1 & 50.4 & 22.7 & 13.73 & 4.2 \\
\hline
\end{tabular*}
\end{table}

\begin{table}
\tablewidth=280pt
\caption{Shapiro--Wilk test under Condition \protect\ref{A3}}\label{tabl2}
\begin{tabular*}{\tablewidth}{@{\extracolsep{4in minus 4in}}llll@{}}
\hline
& & \multicolumn{2}{l@{}}{Shapiro--Wilk test} \\[-4pt]
& & \multicolumn{2}{l@{}}{\hrulefill}\\
$\alpha_{0}$ & $L$ & $W$ & $p$-value \\
\hline
2 & 2000 & 0.9976 & 0.685 \\
& 3000 & 0.9978 & 0.667 \\
& 4000 & 0.9983 & 0.373 \\
[4pt]
3 & 2000 & 0.9976 & 0.691 \\
& 3000 & 0.9980 & 0.842 \\
& 4000 & 0.9985 & 0.945 \\
[4pt]
4 & 2000 & 0.9987 & 0.670 \\
& 3000 & 0.998 & 0.286 \\
& 4000 & 0.9985 & 0.578 \\
\hline
\end{tabular*}
\end{table}

\begin{figure}[b]

\includegraphics{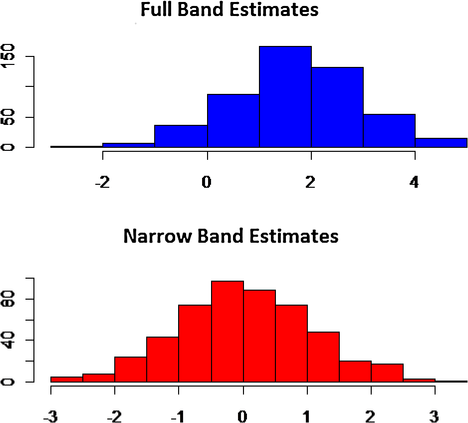}

\caption{Comparison among biased and narrow estimates ($\kappa=1$,
$L=2000$, $\alpha_{0}=3$), under Condition \protect\ref{A2}.}
\label{figur2}
\end{figure}

\begin{table}
\def\arraystretch{0.9}
\tablewidth=280pt
\caption{Normalized Narrow bands data, under Condition \protect\ref
{A2}, $\kappa=1$, $\alpha_{0}=4$, $G_{0}=2$}\label{tabl3}
\begin{tabular*}{\tablewidth}{@{\extracolsep{4in minus 4in}}lld{2.4}lll@{}}
\hline
& & & & \multicolumn{2}{l@{}}{Shapiro--Wilk test} \\[-4pt]
& & & & \multicolumn{2}{l@{}}{\hrulefill} \\
$L$ & $L_{1}$ &
\multicolumn{1}{l}{\hspace*{-2pt}Mean} & Var & $W$ & $p$-value\\
\hline
2000 & 1550 & 0.072 & 0.959 & 0.9985 & 0.950 \\
& 1700 & 0.018 & 0.951 & 0.997 &
0.495 \\
& 1850 & -0.016 & 1.004 & 0.9977 & 0.739 \\
[4pt]
3000 & 2400 & 0.092 & 1.130 & 0.9949 & 0.920 \\
& 2600 & 0.072 & 0.928 & 0.9951 &
0.745 \\
& 2800 & -0.02 & 1.06 & 0.9965 & 0.340 \\
[4pt]
4000 & 3250 & 0.006 & 0.985 & 0.9968 & 0.443 \\
& 3500 &
0.004 & 1.097 & 0.998 & 0.834 \\
& 3750 & 0.0007 & 1.073 & 0.9982 & 0.874 \\
\hline
\end{tabular*}
\end{table}

\begin{figure}[b]

\includegraphics{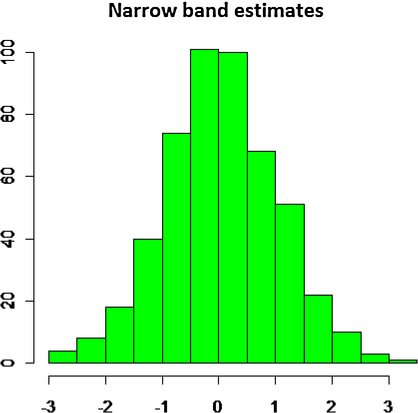}

\caption{Distribution of normalized $ ( \widehat{\alpha
}_{L}-\alpha_{0} ) $ under
Condition \protect\ref{A1}.}
\label{figur3}
\end{figure}

We obtain a mean value $\mathbb{E} ( \widehat{\alpha}_{L}-\alpha_{0} ) =0.040$ and a normalized variance of $0.9918$. Shapiro--Wilk
Gaussianity test gives as result $W=0.9981$ with a $p$-$\mbox{value}=0.8669$.
Table~\ref{tabl4}
compares sample variance, bias and mean squared errors obtained for
simulations with different values of $L$, $\kappa$ and $\alpha_{0}$
with $%
N=5000$ iterations.

\begin{table}
\caption{Sample Variance, Bias and MSE of estimators
$\widehat{\alpha}_{L}$ and $\widehat{\alpha}_{L_{1},L}$ for different
values of $L=1000,2000,5000,10\,000$ and $\kappa=1,2$ ($\alpha_{0}=3$)}\label{tabl4}
\begin{tabular*}{\tablewidth}{@{\extracolsep{\fill}}llllllll@{}}
\hline
%
\multicolumn{1}{@{}l}{$\kappa$} & \multicolumn{1}{l}{Band} & Var &
Bias &
\multicolumn{1}{l}{MSE} & Var & \multicolumn{1}{l}{Bias} &
\multicolumn{1}{l@{}}{MSE} \\
\hline
\multicolumn{4}{@{}l}{$L=1000$} & &
\multicolumn{3}{l@{}}{$L=5000$}\\[4pt]
1 & Full & $7.9\cdot10^{-6}$
& $0.004$ & $2.4\cdot10^{-5}$ & $3.2\cdot10^{-7}$ & $0.0008$ &
$9.7\cdot
10^{-7}$ \\
& Nar. & $1.4\cdot10^{-4}$ & $%
0.001$ & $1.5\cdot10^{-4}$ & $5.4\cdot10^{-5}$ & $0.0003$ & $5.4\cdot
10^{-5}$ \\
[4pt]
2 & \multicolumn{1}{l}{Full} & $8.0\cdot10^{-6}$
& $0.008$ & $7.1\cdot10^{-5}$ & $3.3\cdot10^{-7}$ & $0.002$ &
$6.1\cdot
10^{-6}$ \\
& Nar. & $1.4\cdot
10^{-4}$ & $%
0.002$ & $1.5\cdot10^{-4}$ & $5.3\cdot10^{-5}$ & $0.0006$ & $5.4\cdot
10^{-5}$ \\
%
[6pt]
\multicolumn{4}{@{}l}{$L=2000$} & &
\multicolumn{3}{l@{}}{$L=10\,000$}\\
[6pt]
1 & Full &
$1.9\cdot10^{-6}$ & 0.002 &
$5.8\cdot10^{-6}$ & $8.1\cdot10^{-8}$
& 0.0004 & $2.4\cdot10^{-7}$ \\
& Nar. & $9.6\cdot10^{-5}$ & 0.0005 &
$9.6\cdot10^{-5}$ & $1.3\cdot10^{-5}$ &
$9\cdot10^{-5}$ & $1.3\cdot10^{-5}$ \\
[4pt]
2 & Full & $1.9\cdot10^{-6}$ & 0.004 &
$1.8\cdot10^{-5}$ & $8.1\cdot
10^{-8}$ & 0.0008 & $2.4\cdot
10^{-7}$ \\
& Nar. & $9.6\cdot10^{-5}$
& 0.001 & $9.6\cdot10^{-5}$ & $1.3\cdot10^{-5}$ &
0.0002 & $1.3\cdot10^{-5}$\\
\hline
\end{tabular*}
\end{table}

The simulations show that full-band estimators is characterized by a smaller
MSE with respect to the corresponding narrow band estimators obtained
on the
same data sets, due to the smallest value of the variance. Hence, full band
estimates seem to be more efficient than the narrow band ones, although they
appear to be more robust. Note that for the sake of the brevity we report
only the data concerning $\alpha_{0}=3$, because data obtained for
$\alpha_{0}=2,4$ lead to very similar results.

%
\begin{figure}[b]

\includegraphics{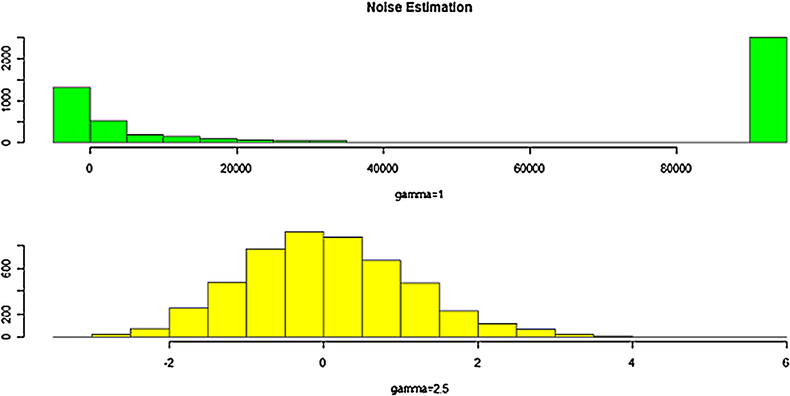}

\caption{Distribution of normalized $\widetilde{\alpha}-\alpha_{0}$
($\alpha_{0}=3$, $L=1000$) in presence of noise ($\gamma=1$ and~$2.5$).}
\label{figur4}
\end{figure}

In Figure \ref{figur4}, we report results on simulations (iterated $N=5000$ times)
which take in account also the presence of the noise, using $\alpha_{0}=3$,
$L=1000$ and by varying the value of $\gamma$. In these simulations, we
consider four cases. In the cases $\gamma=5$ and $\gamma=3$, the
results obtained put in evidence that in the case $\gamma>\alpha_{0}$ the
noise does not affect the signal detected (we omit these results in the
figure). If instead $\gamma=2.5$, we obtain the convergence of the
estimator to $\alpha_{0}$ with the rate of convergence as described in
Theorem~\ref{theonoise}: in this case $\mathbb{E} ( \widetilde
{\alpha}%
_{L} ) =0.005$, while the variance of the normalized $\widetilde
{\alpha
}_{L}$ corresponds to 1.22. Shapiro--Wilk normality test provides $W=0.9919$
with $p$-$\mbox{value}=2.68\cdot10^{-16}$. Finally, if $\gamma=1$ (and then
$\gamma
<\alpha_{0}-1$) the estimate computed assumes mainly values close to $%
\alpha_{\max}$, the highest value which is allowed by the computational
point of view (in the figure $\alpha_{\max}=50$), hence it seems to
diverge.\vadjust{\goodbreak}

\section{Conclusions} \label{conclusions}

We view this paper as a first contribution in an area which deserves much
further research, that is, the investigation of asymptotic properties for
parametric estimators on a single realization of an isotropic random field
on the sphere. As mentioned earlier, an enormous amount of applied papers
have focussed on this issue, especially in a Cosmological framework,
but no
rigorous results seem currently available. Our results suggest that
consistency and asymptotic Gaussianity are feasible for spectral index
estimators, the rate of convergence being $L/\log L$; these estimates are
centred on zero in ``parametric''
circumstances, that is, where the correct model being provided for
$C_{l}$ up to
a factor $\RMo(\frac{1}{l})$. When the latter assumption fails, alternatively,
narrow-band estimates can be entertained; these estimates ensure convergence
to a zero-mean Gaussian distribution, with a slightly slower
convergence rate.

Many questions are left open by these results. The first we mention is
the characterization of a whole class of parameters for which
asymptotic Gaussianity and consistency may continue to hold. More
challenging is the possibility to relax the Gaussian assumption and
consider more general, finite-variance isotropic Gaussian fields. In
this respect, results in \cite{marpec2} suggest that the Gaussianity
assumption may indeed play a crucial role, as high-frequency
consistency and Gaussianity seem very tightly related, for instance,
when considering the asymptotic behavior of the angular power spectrum.
It seems also important to explore the connection between the spherical
estimates we have been considering and fixed-domain asymptotic results
for Matern-type covariances, as discussed on $\mathbb
{R}%
^{d}$ by \cite{loh,anderes,wang,wu} and others.
Likewise, the high-frequency behaviour of Bayesian estimates definitely
deserves some investigation in this framework, especially considering the
growing interest for Bayesian techniques in the astrophysical community.

For future work, we aim at relaxing some of the assumptions introduced in
this paper to make these techniques more directly applicable on existing
datasets. The harmonic estimates we have been focussing on require the
observation of the random field on the full sphere. This condition often
fails in practice: for instance, in a Cosmological framework large regions
of the sky are not observable, because they are masked by Foreground sources
such as the Milky Way. In ongoing research (see \cite{durastanti}), we are
hence considering a Whittle-type estimator based on spherical wavelets
(needlets, see \cite{npw1,mpbb08,bkmpAoS}), rather than
standard Fourier analysis. These estimates have, however, a larger
asymptotic variance than the Fourier methods considered here; in a sense,
this is an instance of the standard trade-off between robustness and
efficiency. Thus, the material in the present paper presents a
benchmark for
optimal procedures under favourable experimental circumstances, and the
right starting point for further developments under more challenging
experimental set-ups.

\begin{appendix}\label{app}
\section*{Appendix}
\subsection*{Consistency results}

\begin{pf*}{Proof of Theorem \ref{consistenza}}
To establish consistency, we shall resort to a
technique developed by \cite{brillinger} and \cite{Robinson}. In particular,
let us now write%
\begin{eqnarray*}
\Delta R_{L}(\alpha,\alpha_{0})&=&R_{L}(
\alpha)-R_{L}(\alpha_{0})
\\
&=&\log\frac{\widehat{G}(\alpha)}{G(\alpha)}-\log\frac{\widehat
{G}(\alpha_{0})}{G(\alpha_{0})}-\frac{(\alpha-\alpha_{0})}{\sum_{l=1}^{L}(2l+1)}%
\sum_{l=1}^{L}(2l+1)\log l\\
&&{}+\log
\frac{G(\alpha)}{G(\alpha_{0})},
\end{eqnarray*}
where%
\begin{eqnarray*}
G(\alpha) &=&\frac{1}{\sum_{l=1}^{L}(2l+1)}\sum_{l=1}^{L}(2l+1)
\frac{%
G_{0}l^{-\alpha_{0}}}{l^{-\alpha}},\qquad G(\alpha_{0})=G_{0},
\\
\log\frac{G(\alpha)}{G(\alpha_{0})} &=&\log \Biggl\{ \frac{1}{%
\sum_{l=1}^{L}(2l+1)}\sum
_{l=1}^{L}(2l+1)l^{\alpha-\alpha_{0}} \Biggr\}
\end{eqnarray*}
so that
%
\begin{eqnarray}
\Delta R_{L}(\alpha,\alpha_{0})&=&U_{L}(
\alpha,\alpha_{0})-T_{L}(\alpha,\alpha_{0}),
\nonumber
\\
\label{intrigila1} U_{L} ( \alpha,\alpha_{0} ) &=&-
\frac{(\alpha-\alpha_{0})}{
\sum_{l=1}^{L}(2l+1)}\sum_{l=1}^{L}(2l+1)\log
l+\log\frac{G(\alpha)}{%
G(\alpha_{0})},
\\
\label{intrigila2} T_{L} ( \alpha,\alpha_{0} ) &=&\log
\frac{\widehat{G}(\alpha_{0})%
}{G(\alpha_{0})}-\log\frac{\widehat{G}(\alpha)}{G(\alpha)}.
\end{eqnarray}

The proof is then completed with the aid of the auxiliary Lemmas \ref%
{consistency1}, \ref{consistency2} that we shall discuss below. Indeed%
\begin{eqnarray*}
\Pr \bigl\{ \llvert \widehat{\alpha}_{L}-\alpha_{0}
\rrvert >\varepsilon \bigr\} &\leq&\Pr \Bigl\{ \inf_{\llvert \alpha-\alpha
_{0}\rrvert >\varepsilon}\Delta
R_{L}(\alpha,\alpha_{0})\leq 0 \Bigr\}
\\
&\leq&\Pr \Bigl\{ \inf_{\llvert \alpha-\alpha_{0}\rrvert
>\varepsilon} \bigl[ U_{L}(\alpha,
\alpha_{0})-T_{L}(\alpha,\alpha_{0})%
\bigr] \leq0 \Bigr\}.
\end{eqnarray*}
For $\alpha_{0}-\alpha<2$ the previous probability is bounded by, for
any $%
\delta>0$%
\[
\leq\Pr \Bigl\{ \inf_{\llvert \alpha-\alpha_{0}\rrvert
>\varepsilon}U_{L} ( \alpha,
\alpha_{0} ) \leq\delta \Bigr\} +\Pr \Bigl\{ \sup_{\llvert \alpha-\alpha_{0}\rrvert
>\varepsilon
}T_{L}
( a,\alpha_{0} ) >0 \Bigr\}
\]
and%
\[
\lim_{L\rightarrow\infty}\Pr \Bigl\{ \sup_{\llvert \alpha-\alpha
_{0}\rrvert >\varepsilon}T_{L} ( a,
\alpha_{0} ) >0 \Bigr\} =0
\]
from Lemma \ref{consistency2}, while from Lemma \ref{consistency1} there
exist $\delta_{\varepsilon}=(1+\varepsilon/2)-\log(1+\varepsilon/2)-1>0$
such that
\[
\lim_{L\rightarrow\infty}\Pr \Bigl\{ \inf_{\llvert \alpha-\alpha
_{0}\rrvert >\varepsilon}U_{L} ( \alpha,
\alpha_{0} ) \leq \delta_{\varepsilon} \Bigr\} =0.
\]
For $\alpha_{0}-\alpha=2$ or $\alpha_{0}-\alpha>2$ the same result is
obtained by dividing $\Delta R_{L}(\alpha,\alpha_{0})$ by,
respectively, $%
\log\log L$ or $\log L$ and then resorting again to Lemmas \ref%
{consistency1}, \ref{consistency2}.

Now note that%
\begin{eqnarray*}
\widehat{G}(\widehat{\alpha}_{L})-G_{0} &=&
\frac{1}{\sum_{l=1}^{L}(2l+1)}%
\sum_{l=1}^{L}(2l+1)
\frac{\widehat{C}_{l}}{l^{-\widehat{\alpha
}_{L}}}\\
&&{}-\frac{%
1}{\sum_{l=1}^{L}(2l+1)}\sum_{l=1}^{L}(2l+1)
\frac{G_{0}l^{-\alpha
_{0}}}{%
l^{-\alpha_{0}}}
\\
&=&\frac{1}{\sum_{l=1}^{L}(2l+1)}\sum_{l=1}^{L}(2l+1)G_{0}l^{- (
\alpha_{0}-\widehat{\alpha}_{L} ) }
\biggl\{ \biggl( \frac{\widehat
{C}%
_{l}}{G_{0}l^{-\alpha_{0}}}-1 \biggr) + \bigl( 1-l^{ ( \alpha_{0}-%
\widehat{\alpha}_{L} ) } \bigr)
\biggr\}.
\end{eqnarray*}
Clearly:%
\begin{eqnarray*}
\bigl\llvert \widehat{G}(\widehat{\alpha}_{L})-G_{0}\bigr
\rrvert &\leq &\Biggl\llvert \frac{1}{\sum_{l=1}^{L}(2l+1)}%
\sum
_{l=1}^{L}(2l+1)G_{0}l^{- ( \alpha_{0}-\widehat{\alpha
}_{L} )
}
\biggl\{ \biggl( \frac{\widehat{C}_{l}}{G_{0}l^{-\alpha_{0}}}-1 \biggr) \biggr\} \Biggr\rrvert
\\
&&{}+\Biggl\llvert \frac{G_{0}}{\sum_{l=1}^{L}(2l+1)}\sum_{l=1}^{L}(2l+1)%
\bigl( l^{- ( \alpha_{0}-\widehat{\alpha}_{L} ) }-1 \bigr) \Biggr\rrvert
\\
&=&\llvert G_{A}\rrvert +\llvert G_{B}\rrvert ,
\end{eqnarray*}
so that
\[
\Pr \bigl( \bigl\llvert \widehat{G}(\widehat{\alpha}_{L})-G_{0}
\bigr\rrvert \geq\varepsilon \bigr) \leq\Pr \biggl( \llvert G_{A}
\rrvert \geq \frac{\varepsilon}{2} \biggr) +\Pr \biggl( \llvert G_{B}
\rrvert \geq \frac{\varepsilon}{2} \biggr).
\]
Observe that:%
\begin{eqnarray*}
\Pr \biggl\{ \llvert G_{A}\rrvert \geq\frac{\varepsilon}{2}%
\biggr\} &\leq&\Pr \biggl\{ \biggl[ \llvert G_{A}\rrvert \geq
\frac{
\varepsilon}{2} \biggr] \cap \biggl[ \llvert \alpha_{0}-\widehat {
\alpha}%
_{L}\rrvert <\frac{1}{3} \biggr] \biggr\}
\\
&&{}+\Pr \biggl\{ \llvert \alpha_{0}-\widehat{\alpha}_{L}
\rrvert \geq \frac{1}{3} \biggr\}
\\
&\leq&\Pr \Biggl\{ \Biggl[ \frac{1}{\sum_{l=1}^{L}(2l+1)}%
\sum
_{l=1}^{L}(2l+1)G_{0}l^{1/3}
\biggl\llvert \frac{\widehat{C}_{l}}{%
G_{0}l^{-\alpha_{0}}}-1\biggr\rrvert \geq\varepsilon \Biggr] \Biggr\}
+\RMo_{L}(1)
\\
&\leq&\frac{1}{\varepsilon}\frac{1}{\sum_{l=1}^{L}(2l+1)}%
\sum
_{l=1}^{L}(2l+1)G_{0}l^{1/3}
\mathbb{E}\biggl\llvert \frac{\widehat
{C}_{l}}{%
G_{0}l^{-\alpha_{0}}}-1\biggr\rrvert +\RMo_{L}(1)
\\
&\leq&\frac{C}{\varepsilon}\frac{1}{\sum_{l=1}^{L}(2l+1)}%
\sum
_{l=1}^{L}(2l+1)G_{0}l^{1/3}l^{-1/2}+
\RMo_{L}(1)
\\
&=&\frac{C}{\varepsilon}\frac{L^{11/6}}{\sum_{l=1}^{L}(2l+1)}%
+\RMo_{L}(1)=
\RMo_{L}(1).
\end{eqnarray*}
As far as the second term is concerned, we have, for a suitably small $%
\delta>0$:%
\begin{eqnarray*}
\Pr \biggl( \llvert G_{B}\rrvert \geq\frac{\varepsilon}{2} \biggr) &=&
\Pr \biggl( \biggl[ \llvert G_{B}\rrvert \geq\frac{\varepsilon
}{2}%
\biggr] \cap \bigl[ \log l ( \alpha_{0}-\widehat{
\alpha}_{L} ) \bigr] <\delta \biggr) \\
&&{}+\Pr \bigl( \log l (
\alpha_{0}-\widehat {\alpha }_{L} ) \geq\delta \bigr)
\\
&=&\Pr \biggl( \biggl[ \llvert G_{B}\rrvert \geq
\frac{\varepsilon
}{2}%
\biggr] \cap \bigl[ \log l ( \alpha_{0}-
\widehat{\alpha}_{L} ) \bigr] <\delta \biggr) +\RMo_{L}(1)
\end{eqnarray*}
and using $\llvert  \RMe^{-x}-1\rrvert \leq x$ for $0\leq x\leq1$, we
obtain%
\begin{eqnarray*}
\bigl\llvert l^{- ( \alpha_{0}-\widehat{\alpha}_{L} )
}-1\bigr\rrvert &=&\bigl\llvert \exp\bigl(-\log l (
\alpha_{0}-\widehat {\alpha}%
_{L} )\bigr) -1\bigr
\rrvert \leq\log l\llvert \alpha_{0}-\widehat {\alpha}%
_{L}
\rrvert ,\\[-27pt]
\end{eqnarray*}
\begin{eqnarray*}
&&\Pr \biggl( \biggl[ \llvert G_{B}\rrvert \geq\frac{\varepsilon
}{2}%
\biggr] \cap \bigl[ \log l ( \alpha_{0}-\widehat{
\alpha}_{L} ) \bigr] <\delta \biggr)
\\
&&\quad\leq\Pr \Biggl( \frac{G_{0}}{\sum_{l=1}^{L}(2l+1)}\sum_{l=1}^{L}(2l+1)
\bigl\llvert \bigl( l^{- ( \alpha_{0}-\widehat{\alpha}_{L} )
}-1 \bigr) \bigr\rrvert \geq
\frac{\varepsilon}{2}\cap \bigl[ \log l ( \alpha_{0}-\widehat{
\alpha}_{L} ) \bigr] <\delta \Biggr)
\\
&&\quad\leq\frac{1}{\varepsilon}\mathbb{E} \Biggl\{ \frac{G_{0}}{%
\sum_{l=1}^{L}(2l+1)}\sum
_{l=1}^{L}(2l+1)\log l\llvert
\alpha_{0}-%
\widehat{\alpha}_{L}\rrvert \Biggr\}
\\
&&\quad\leq\frac{C}{\varepsilon}\frac{G_{0}}{\sum_{l=1}^{L}(2l+1)}%
\sum
_{l=1}^{L}(2l+1)\log l\frac{\log L}{L}=
\RMo_{L}(1),
\end{eqnarray*}
where we have used
\[
\mathbb{E}\llvert \alpha_{0}-\widehat{\alpha}_{L}\rrvert
\leq \bigl\{ \mathbb{E}\llvert \alpha_{0}-\widehat{
\alpha}_{L}\rrvert^{2} \bigr\}^{1/2}=\RMO\biggl(
\frac{\log L}{L}\biggr),
\]
which under Condition \ref{A1} will be established in the proof of
Theorem %
\ref{clt0}.
\end{pf*}

The first auxiliary result we shall need concerns $G,\widehat{G}$ and
their $%
k$th order derivatives $G_{k},\widehat{G}_{k}$, that is,
\begin{eqnarray*}
\widehat{G}_{k} ( \alpha ) &=&\frac{1}{\sum_{l=1}^{L}(2l+1)}%
\sum
_{l=1}^{L}(2l+1) \bigl(
\log^{k}l \bigr) \frac{\widehat{C}_{l}}{%
l^{-\alpha}},\qquad k=0,1,2,\ldots,
\\
G_{k} ( \alpha ) &=&\frac{1}{\sum_{l=1}^{L}(2l+1)}%
\sum
_{l=1}^{L}(2l+1) \bigl( \log^{k}l \bigr)
\frac{G_{0}l^{-\alpha
_{0}}}{%
l^{-\alpha}},\qquad k=0,1,2,\ldots,
\end{eqnarray*}
where $\widehat{G}_{0}(\alpha)=\widehat{G}(\alpha)$ and $G_{0}(\alpha
)=G(\alpha)$ defined as above.
%
\begin{lemma}
\label{lemmagiovedi} Under Condition \ref{A1}, for all $2>\alpha_{0}-\alpha>\varepsilon>0$, as $L\rightarrow\infty$, we have%
\[
\sup_{\alpha}\biggl\llvert \log\frac{\widehat{G}_{k}(\alpha
)}{G_{k}(\alpha)}%
\biggr
\rrvert =\RMo_{p}(1).
\]
On the other hand, if $\alpha_{0}-\alpha\geq2$,
\[
\sup_{\alpha}\biggl\llvert \log\frac{\widehat{G}_{k}(\alpha
)}{G_{k}(\alpha)}%
\biggr
\rrvert =\RMO_{p}(1).
\]
\end{lemma}
\begin{pf}
Let us first focus on the case where $\alpha-\alpha_{0}>-2$. For clarity
of exposition, we start from a simplified parametric version of
Condition %
\ref{A}, that is, we assume that we have exactly%
\[
C_{l}(\vartheta)=C_{l}(G_{0},
\alpha_{0})=G_{0}l^{-\alpha_{0}}.
\]
Let us write first%
\begin{eqnarray*}
\frac{\widehat{G}_{k}(\alpha)}{G_{k}(\alpha)}-1
&=&\frac{(\sum_{l=1}^{L}(2l+1) ( \log l )^{k}{ 
\widehat{C}_{l}}/{l^{-\alpha}})/
({\sum_{l=1}^{L}(2l+1)})}{(%
\sum_{l=1}^{L}(2l+1) ( \log l )^{k}{G_{0}l^{-\alpha
_{0}}}/{%
l^{-\alpha}})/({\sum_{l=1}^{L}(2l+1)})}-1
\\
&=&\frac{\sum_{l=1}^{L}(2l+1) ( \log l )^{k}G_{0}l^{\alpha
-\alpha
_{0}} \{ {\widehat{C}_{l}}/({G_{0}l^{-\alpha_{0}}})-1 \} }{%
\sum_{l=1}^{L}(2l+1) ( \log l )^{k}G_{0}l^{\alpha-\alpha
_{0}}}.
\end{eqnarray*}
Fixed $0<\beta<\frac{1}{2}$, we have, for all $l$:
\begin{eqnarray*}
&&\Pr \biggl( \biggl\llvert
\frac{\sum_{l=1}^{L}(2l+1)G_{0}l^{\alpha-\alpha _{0}} ( \log l )^{k} \{
{\widehat {C}_{l}}/({G_{0}l^{-\alpha _{0}}})-1 \}
}{\sum_{l=1}^{L}(2l+1) ( \log l
)^{k}G_{0}l^{\alpha-\alpha_{0}}}\biggr\rrvert > \delta_{\varepsilon }
\biggr)
\\
&&\quad\leq\Pr \biggl( L^{\beta}\biggl\llvert \frac{\sum_{l=1}^{L}\sqrt{(2l+1)}
( \log l )^{k}l^{\alpha-\alpha_{0}}}{\sum_{l=1}^{L}(2l+1) (
\log l )^{k}l^{\alpha-\alpha_{0}}}\biggr
\rrvert \frac{\sup_{l}\sqrt{%
(2l+1)}\llvert {\widehat{C}_{l}}/({G_{0}l^{-\alpha_{0}}})-1\rrvert
}{L^{\beta}}>\delta_{\varepsilon} \biggr)
\\
&&\quad\leq\Pr \biggl( \sup_{l}\sqrt{(2l+1)}\biggl\llvert
\frac{\widehat
{C}_{l}}{%
G_{0}l^{-\alpha_{0}}}-1\biggr\rrvert >\delta_{\varepsilon}L^{\beta
}
\biggr),
\end{eqnarray*}
because
\[
L^{\beta}\frac{\sum_{l=1}^{L}\sqrt{(2l+1)}
( \log l )^{k}l^{\alpha-\alpha_{0}}}{\sum_{l=1}^{L}(2l+1)
( \log l )^{k}l^{\alpha-\alpha_{0}}}=C\frac{L^{\beta+{3}/{2}+\alpha-\alpha
_{0}}\log^{k}L}{L^{2+\alpha-\alpha_{0}}\log^{k}L}=CL^{\beta
-1/2}=
\RMo(1).
\]

Now%
\begin{eqnarray*}
&&\Pr \biggl\{ \sup_{l}\sqrt{(2l+1)}\biggl\llvert
\frac{\widehat{C}_{l}}{%
G_{0}l^{-\alpha_{0}}}-1\biggr\rrvert >\delta_{\varepsilon}L^{\beta
} \biggr
\}
\\
&&\quad\leq L\max_{l}\Pr \biggl\{ \sqrt{(2l+1)}\biggl\llvert
\frac{\widehat
{C}_{l}}{%
G_{0}l^{-\alpha_{0}}}-1\biggr\rrvert >\delta_{\varepsilon}L^{\beta
} \biggr
\}
\end{eqnarray*}
and%
\[
\Pr \biggl\{ \sqrt{(2l+1)}\biggl\llvert \frac{\widehat
{C}_{l}}{G_{0}l^{-\alpha
_{0}}}-1\biggr\rrvert >
\delta_{\varepsilon}L^{\beta} \biggr\} \leq C\frac{%
\mathbb{E} [ \sqrt{(2l+1)}\llvert {\widehat{C}_{l}}/({%
G_{0}l^{-\alpha_{0}}})-1\rrvert  ]^{M}}{\delta_{\varepsilon
}^{M}L^{M\beta}}=\RMO
\bigl(L^{-M\beta}\bigr),
\]
uniformly in $l$, see, for instance, \cite{marpecbook}, such that
$M>1/\beta$. Hence,
\[
\Pr \biggl\{ \sup_{l}\sqrt{(2l+1)}\biggl\llvert \frac{\widehat{C}_{l}}{%
G_{0}l^{-\alpha_{0}}}-1
\biggr\rrvert >\delta_{\varepsilon}L^{\beta
} \biggr\} =\RMO
\bigl(L^{1-M\beta}\bigr)=\RMo_{L}(1).
\]
For the general semiparametric case, the only difference is to be found in
the expressions for $\mathbb{E}\widehat{C}_{l}$, which under Condition
\ref%
{A1} becomes%
\[
\mathbb{E}\widehat{C}_{l}=G_{0}l^{-\alpha_{0}} \bigl( 1+
\RMO \bigl( l^{-1} \bigr) \bigr),
\]
where the bound $\RMO ( l^{-1} ) $ is uniform over $\alpha$ by
assumption. As before, we hence obtain%
\begin{eqnarray*}
&& \frac{\widehat{G}_{k}(\alpha)}{G_{k}(\alpha)}-1
\\
&&\quad=\frac{\sum_{l=1}^{L}(2l+1) ( \log l )^{k}{\widehat
{C}_{l}}/{%
l^{-\alpha}}-\sum_{l=1}^{L}(2l+1) ( \log l )^{k}{%
G_{0}l^{-\alpha_{0}}}/{l^{-\alpha}}}{\sum_{l=1}^{L}(2l+1) ( \log
l )^{k}{G_{0}l^{-\alpha_{0}}}/{l^{-\alpha}}}
\\
&&\quad=\frac{\sum_{l=1}^{L}(2l+1) ( \log l )^{k}G_{0}l^{\alpha
-\alpha
_{0}} \{ {\widehat{C}_{l}}/({G_{0}l^{-\alpha_{0}}})-{\mathbb
{E}%
\widehat{C}_{l}}/({G_{0}l^{-\alpha_{0}}}) \} }{\sum_{l=1}^{L}(2l+1) (
\log l )^{k}G_{0}l^{\alpha-\alpha_{0}}}
\\
&&\qquad{}+\frac{\sum_{l=1}^{L}(2l+1) ( \log l )^{k}G_{0}l^{\alpha
-\alpha
_{0}} \{ \RMO({1}/{l}) \} }{\sum_{l=1}^{L}(2l+1) ( \log
l )^{k}G_{0}l^{\alpha-\alpha_{0}}}.
\end{eqnarray*}
The second summand is immediately observed to be $\RMO(\frac{1}{L})$. By the
same argument as before, for $0<\beta<\frac{1}{2}$, we have, for all
$l$:
\begin{eqnarray*}
&&\Pr \biggl\{ \biggl\llvert \frac{\sum_{l=1}^{L}(2l+1)
( \log l )^{k}G_{0}l^{\alpha-\alpha_{0}} \{ {\widehat{C}_{l}}/({%
G_{0}l^{-\alpha_{0}}})-{\mathbb{E}\widehat{C}_{l}}/({G_{0}l^{-\alpha
_{0}}%
}) \} }{\sum_{l=1}^{L}(2l+1) ( \log l )^{k}G_{0}l^{\alpha
-\alpha_{0}}}\biggr\rrvert >
\delta_{\varepsilon} \biggr\}
\\
&&\quad\leq\Pr \biggl\{ \sup_{l}\sqrt{(2l+1)}\frac{\mathbb{E}\widehat
{C}_{l}}{%
G_{0}l^{-\alpha_{0}}}\biggl
\llvert \frac{\widehat{C}_{l}}{\mathbb
{E}\widehat{C}%
_{l}}-1\biggr\rrvert >\delta_{\varepsilon}L^{\beta}
\biggr\}
\\
&&\quad\leq\Pr \biggl\{ \sup_{l}\sqrt{(2l+1)} \biggl\{ 1+\RMO\biggl(
\frac{1}{l}\biggr) \biggr\} \biggl\llvert \frac{\widehat{C}_{l}}{\mathbb{E}\widehat{C}_{l}}-1\biggr\rrvert
>\delta_{\varepsilon}L^{\beta} \biggr\}.
\end{eqnarray*}

The rest of the proof is analogous to the argument we provided before, and
hence omitted.

For the case where $\alpha_{0}-\alpha\geq2$, it suffices to note that
\[
\frac{\widehat{G}_{k}(\alpha)}{G_{k}(\alpha)}=\frac{
(\sum_{l=1}^{L}(2l+1) ( \log l )^{k}{
\widehat{C}_{l}}/{l^{-\alpha}})/
({%
\sum_{l=1}^{L}(2l+1)})}{
(\sum_{l=1}^{L}(2l+1) ( \log l )^{k}{G_{0}l^{-\alpha
_{0}}}/{%
l^{-\alpha}})/({\sum_{l=1}^{L}(2l+1)})}>0\qquad\mbox{with probability 1}
\]
and%
\[
\mathbb{E}\frac{\widehat{G}_{k}(\alpha)}{G_{k}(\alpha)}=\frac{
(\sum_{l=1}^{L}(2l+1) ( \log l )^{k}({
G_{0}l^{-\alpha_{0}}}/{l^{-\alpha}}) \{ 1+\RMO({1}/{l}) \})/
({%
\sum_{l=1}^{L}(2l+1)})
}{
(\sum_{l=1}^{L}(2l+1) ( \log l )^{k}{%
G_{0}l^{-\alpha_{0}}}/{l^{-\alpha}})/
({\sum_{l=1}^{L}(2l+1)})}=\RMO(1).\quad
\]
\upqed
\end{pf}

We are now in the position to establish the asymptotic behavior of $%
U_{L} ( \alpha,\alpha_{0} ) $ in (\ref{intrigila1}), for
which we
have the following:
%
\begin{lemma}
\label{consistency1}For all $2>\alpha_{0}-\alpha>\varepsilon>0$, we
have that%
\begin{eqnarray*}
&& \lim_{L\rightarrow\infty} \Biggl\{ -\frac{(\alpha-\alpha_{0})}{%
\sum_{l=1}^{L}(2l+1)}\sum
_{l=1}^{L}(2l+1)\log l+\log\frac{G(\alpha)}{%
G(\alpha_{0})}
\Biggr\}
\\
&&\quad=\lim_{L\rightarrow\infty} \Biggl[ \log \Biggl\{ \frac{1}{%
\sum_{l=1}^{L}(2l+1)}\sum
_{l=1}^{L}(2l+1)l^{\alpha-\alpha_{0}} \Biggr\} -
\frac{(\alpha-\alpha_{0})}{\sum_{l=1}^{L}(2l+1)}\sum_{l=1}^{L}(2l+1)\log
l%
\Biggr]
\\
&&\quad=\bigl(1+(\alpha-\alpha_{0})/2\bigr)-\log\bigl(1+(\alpha-
\alpha_{0})/2\bigr)-1>\delta_{\varepsilon}>0.
\end{eqnarray*}
Moreover, if $\alpha_{0}-\alpha=2$,
\[
\lim_{L\rightarrow\infty}\frac{1}{\log\log L} \Biggl\{ -\frac{(\alpha
-\alpha_{0})}{\sum_{l=1}^{L}(2l+1)}\sum
_{l=1}^{L}(2l+1)\log l+\log
\frac{%
G(\alpha)}{G(\alpha_{0})} \Biggr\} =1>0
\]
and for $\alpha_{0}-\alpha>2$,
\begin{eqnarray*}
&&
\lim_{L\rightarrow\infty}\frac{1}{\log L} \Biggl\{ -\frac{(\alpha
-\alpha_{0})}{\sum_{l=1}^{L}(2l+1)}\sum
_{l=1}^{L}(2l+1)\log l+\log
\frac
{G(\alpha)%
}{G(\alpha_{0})} \Biggr\}\\
&&\quad =\alpha_{0}-\alpha-2>0.
\end{eqnarray*}
\end{lemma}
\begin{pf}
Consider first the case $\alpha-\alpha_{0}>-2$%
\begin{eqnarray*}
&&\log \Biggl\{ \frac{1}{\sum_{l=1}^{L}(2l+1)}\sum_{l=1}^{L}(2l+1)l^{\alpha
-\alpha_{0}}
\Biggr\}
\\
&&\quad=\log \Biggl\{ \frac{(1+(\alpha-\alpha_{0})/2)}{L^{\alpha-\alpha
_{0}}\sum_{l=1}^{L}(2l+1)}\sum_{l=1}^{L}(2l+1)l^{\alpha-\alpha
_{0}}
\Biggr\} \\
&&\qquad{}-\log\bigl(1+(\alpha-\alpha_{0})/2\bigr)+(\alpha-
\alpha_{0})\log L,
\end{eqnarray*}
where%
\begin{eqnarray*}
\frac{(1+(\alpha-\alpha_{0})/2)}{L^{\alpha-\alpha
_{0}}\sum_{l=1}^{L}(2l+1)}\sum_{l=1}^{L}(2l+1)l^{\alpha-\alpha_{0}}-1
&=&\RMo_{L}(1),
\end{eqnarray*}
whence
\begin{eqnarray*}
\log \Biggl\{ \frac{(1+(\alpha-\alpha_{0})/2)}{L^{\alpha-\alpha
_{0}}\sum_{l=1}^{L}(2l+1)}\sum_{l=1}^{L}(2l+1)l^{\alpha-\alpha
_{0}}
\Biggr\} &=&\RMo_{L}(1).
\end{eqnarray*}
Thus,
\begin{eqnarray*}
&&\log \Biggl\{ \frac{1}{\sum_{l=1}^{L}(2l+1)}\sum_{l=1}^{L}(2l+1)l^{\alpha
-\alpha_{0}}
\Biggr\} -\frac{(\alpha-\alpha_{0})}{\sum_{l=1}^{L}(2l+1)}%
\sum_{l=1}^{L}(2l+1)
\log l
\\
&&\quad=-\log\bigl(1+(\alpha-\alpha_{0})/2\bigr)+(\alpha-
\alpha_{0})\log L\\
&&\qquad{}-\frac{%
(\alpha-\alpha_{0})}{\sum_{l=1}^{L}(2l+1)}\sum
_{l=1}^{L}(2l+1)\log l+\RMo_{L}(1)
\\
&&\quad=\frac{(\alpha-\alpha_{0})}{\sum_{l=1}^{L}(2l+1)}\sum_{l=1}^{L}(2l+1)
(%
\log L-\log l)
\\
&&\qquad{}-\frac{(\alpha-\alpha_{0})}{2}+\frac{(\alpha-\alpha_{0})}{2}-\log \bigl(1+(\alpha-
\alpha_{0})/2\bigr)+\RMo_{L}(1).
\end{eqnarray*}
Now%
\begin{eqnarray*}
&&\frac{(\alpha-\alpha_{0})}{\sum_{l=1}^{L}(2l+1)}\sum_{l=1}^{L}(2l+1)
(%
\log L-\log l)-\frac{(\alpha-\alpha_{0})}{2}
\\
&&\quad=-2(\alpha-\alpha_{0})\int_{0}^{1}x
\log x\mrmd x-\frac{(\alpha-\alpha_{0})}{%
2}+\RMo_{L}(1)=\RMo_{L}(1),
\end{eqnarray*}
because%
\[
\int_{0}^{1}x\log x\mrmd x= \biggl[
\frac{x^{2}}{2}\log x \biggr]_{0}^{1}-%
\int
_{0}^{1}\frac{x^{2}}{2}\frac{1}{x}
\mrmd x=-\frac{1}{4}.
\]

We have hence proved that%
\begin{eqnarray*}
&&\log \Biggl\{ \frac{1}{\sum_{l=1}^{L}(2l+1)}\sum_{l=1}^{L}(2l+1)l^{\alpha
-\alpha_{0}}
\Biggr\} -\frac{(\alpha-\alpha_{0})}{\sum_{l=1}^{L}(2l+1)}%
\sum_{l=1}^{L}(2l+1)
\log l
\\
&&\quad=\bigl(1+ ( \alpha-\alpha_{0} ) /2\bigr)-\log\bigl(1+ (
\alpha -\alpha_{0} ) /2\bigr)-1+\RMo_{L}(1)>0
\end{eqnarray*}
for all $\llvert \alpha-\alpha_{0}\rrvert >\varepsilon$,
$\alpha
-\alpha_{0}>-2$.

Consider now the case $\alpha_{0}-\alpha\geq2$. We can rewrite:
\begin{eqnarray*}
&& -\frac{(\alpha-\alpha_{0})}{\sum_{l=1}^{L}(2l+1)}\sum_{l=1}^{L}(2l+1)
\log l+\log\frac{G(\alpha)}{G(\alpha_{0})}
\\
&&\quad=\log \Biggl\{ \frac{1}{\sum_{l=1}^{L}(2l+1)}\sum_{l=1}^{L}(2l+1)l^{\alpha
-\alpha_{0}}
\Biggr\} -\frac{(\alpha-\alpha_{0})}{\sum_{l=1}^{L}(2l+1)}%
\sum_{l=1}^{L}(2l+1)
\log l
\\
&&\quad= ( \alpha_{0}-\alpha ) \log L \Biggl[ \frac{\log
\sum_{l=1}^{L}(2l+1)l^{- ( \alpha_{0}-\alpha ) }}{ (
\alpha_{0}-\alpha ) \log L}-
\frac{\log\sum_{l=1}^{L} ( 2l+1 )
}{%
( \alpha_{0}-\alpha ) \log L}\\
&&\qquad\hspace*{64pt}{}+\frac{1}{\sum_{l=1}^{L}(2l+1)} 
\sum
_{l=1}^{L}(2l+1)\frac{\log l}{\log L} \Biggr]
\\
&&\quad= ( \alpha_{0}-\alpha ) \log L [ A_{L}+B_{L}+C_{L}
].
\end{eqnarray*}
For the term $A_{L}$:%
\[
\sum_{l=1}^{L}(2l+1)l^{- ( \alpha_{0}-\alpha )
}=c
\sum_{l=1}^{L}l^{1- ( \alpha_{0}-\alpha ) }+
\RMo_{L^{2- (
\alpha_{0}-\alpha ) }} ( 1 ) \rightarrow_{L}c>1,
\]
because $\sum_{l=1}^{L}l^{1- ( \alpha_{0}-\alpha ) }$ is a
convergent series when the exponent $1- ( \alpha_{0}-\alpha )
<-1; $ for $1- ( \alpha_{0}-\alpha ) =-1$, we have $ \{
\sum_{l=1}^{L}l^{1- ( \alpha_{0}-\alpha ) }/\log L \}
\rightarrow1$ and the argument is analogous. Therefore,
\[
( \alpha_{0}-\alpha ) \log L\times [ A_{L} ] =\cases{
\RMO ( \log\log L ), &\quad for $\alpha_{0}-\alpha=2$,
\cr
\RMO(1), &
\quad for $\alpha_{0}-\alpha>2$. }
\]
As far as $B_{L}$ is concerned, we have $\log\sum_{l=1}^{L} (
2l+1 ) =2\log L+\RMo ( \log L ) $, so that:%
\[
\lim_{L\rightarrow\infty}B_{L}=-\frac{2}{ ( \alpha_{0}-\alpha
)
};
\]
finally, simple manipulations and standard properties of the logarithm
(which is a slowly varying function, compare \cite{bingham}) yield
\[
\lim_{L\rightarrow\infty}C_{L}=\lim_{L\rightarrow\infty} \Biggl[
\frac
{1}{%
\sum_{l=1}^{L}(2l+1)}\sum_{l=1}^{L}(2l+1)
\frac{\log l}{\log L} \Biggr] =1.
\]

Summing up, we obtain:%
\[
\lim_{L\rightarrow\infty} \bigl\{ ( \alpha_{0}-\alpha ) \log L [
B_{L}+C_{L} ] \bigr\} =\cases{ 0, &\quad for $
\alpha_{0}-\alpha=2$,
\cr
( \alpha_{0}-\alpha ) -2>0, &
\quad for $\alpha_{0}-\alpha>2$, }
\]
and the claimed result follows.
\end{pf}

In \cite{Robinson} a related computation was given for approximate
Whittle estimates on stationary long memory processes in dimension $d=1$,
that is, the limiting lower bound turned out to be $(1+(\alpha-\alpha_{0}))-\log(1+(\alpha-\alpha_{0}))-1+\RMo_{L}(1)>\delta_{\varepsilon
}$. In
view of this, we conjecture that for general $d$-dimensional spheres the
lower bound will take the form%
\begingroup
\abovedisplayskip=7pt
\belowdisplayskip=7pt
\[
\biggl(1+\frac{(\alpha-\alpha_{0})}{d}\biggr)-\log\biggl(1+\frac{(\alpha-\alpha_{0})}{d}%
\biggr)-1+\RMo_{L}(1)>\delta_{\varepsilon}.
\]

Now we look at $T_{L} ( \alpha,\alpha_{0} ) $, for which we
provide the following
lemma.
%
\begin{lemma}
\label{consistency2} Let $T_{L} ( \alpha,\alpha_{0} ) $ defined
as in (\ref{intrigila2}). Under Condition \ref{A1}, as $L\rightarrow
\infty,
$ we have%
\begin{eqnarray*}
\sup_{\alpha}\bigl\llvert T_{L} ( \alpha,\alpha_{0}
) \bigr\rrvert &=&\RMo_{p}(1)\qquad\mbox{for }\alpha_{0}-
\alpha<2,
\\
\sup_{\alpha}\bigl\llvert T_{L} ( \alpha,\alpha_{0}
) \bigr\rrvert &=&\RMO_{p}(1)\qquad\mbox{for }\alpha_{0}-
\alpha\geq2.
\end{eqnarray*}
\end{lemma}
\begin{pf}
For $\alpha_{0}-\alpha<2$, consider first%
\[
\frac{\widehat{G}(\alpha_{0})}{G(\alpha_{0})}-1=\frac{1}{%
\sum_{l=1}^{L}(2l+1)}\sum_{l=1}^{L}(2l+1)
\biggl(\frac{\widehat{C}_{l}}{%
G_{0}l^{-\alpha_{0}}}-1\biggr),
\]
where we have easily, as $L\rightarrow\infty$,
\begin{eqnarray*}
\mathbb{E} \biggl\{ \frac{\widehat{G}(\alpha_{0})}{G(\alpha_{0})}-1 \biggr\} &=&\frac{1}{\sum_{l=1}^{L}(2l+1)}\sum
_{l=1}^{L}(2l+1) \biggl(
\frac
{G_{0}l^{-\alpha
_{0}} \{ 1+\RMO(l^{-1}) \} }{G_{0}l^{-\alpha_{0}}}-1\biggr)\rightarrow 0,
\\
\operatorname{Var} \biggl\{ \frac{\widehat{G}(\alpha_{0})}{G(\alpha_{0})} \biggr\} &=& \biggl\{ \frac{1}{\sum_{l=1}^{L}(2l+1)} \biggr
\}^{2}\sum_{l=1}^{L}(2l+1)^{2}=
\RMO\biggl(\frac{1}{L}\biggr),
\end{eqnarray*}
whence by Slutzky's lemma%
\[
\biggl\{ \frac{\widehat{G}_{L}(\alpha_{0})}{G_{L}(\alpha_{0})}\stackrel {%
\mathbb{P}} {
\longrightarrow}1 \biggr\} \Rightarrow \biggl\{ \log\frac{%
\widehat{G}(\alpha_{0})}{G(\alpha_{0})}\stackrel{
\mathbb{P}} {%
\longrightarrow}0 \biggr\}.
\]
On the other hand, in view of Lemma \ref{lemmagiovedi}, we have that:%
\[
\sup_{\alpha}\biggl\llvert \log\frac{\widehat{G}(\alpha)}{G(\alpha)}%
\biggr
\rrvert =\RMo_{p}(1),
\]
whence the result follows easily. The proof for $\alpha_{0}-\alpha
\geq2$
is immediate.
\end{pf}
\endgroup

\subsection*{Some integral approximation results}

The following lemma is straightforward.
%
\begin{lemma}
\label{sumlog} Let $L_{1}<L $, then we have%
%
\begin{eqnarray}
\label{slog0} \int_{L_{1}}^{L}2x^{1+s}\mrmd
x&=&\frac{1}{ ( 1+{s}/{2} ) } \bigl( L^{2 ( 1+{s}/{2} ) }-L_{1}^{2 ( 1+{s}/{2} )
}
\bigr);\vadjust{\goodbreak}
\end{eqnarray}
\begin{eqnarray*}
\int_{L_{1}}^{L}2x^{1+s}\log x\mrmd x&=&-
\frac{L^{2 ( 1+{s}/{2} )
}-L_{1}^{2 ( 1+{s}/{2} ) }}{2 ( 1+{s}/{2} )^{2}}+%
\frac{L^{2 ( 1+{s}/{2} ) }\log L-L_{1}^{2 ( 1+
{s}/{2}%
) }\log L_{1}}{ ( 1+{s}/{2} ) };
\nonumber\\
\int_{L_{1}}^{L}2x^{1+s}
\log^{2}x\mrmd x &=&\frac{L^{2 ( 1+{s}/{2}%
) }-L_{1}^{2 ( 1+{s}/{2} ) }}{2 ( 1+{s}/{2}%
)^{3}}-\frac{L^{2 ( 1+{s}/{2} ) }\log
L-L_{1}^{2 ( 1+%
{s}/{2} ) }\log L_{1}}{ ( 1+{s}/{2} )^{2}}
\nonumber\\
&&{}+\frac{L^{2 ( 1+{s}/{2} ) }\log^{2}L-L_{1}^{2 (
1+{s%
}/{2} ) }\log^{2}L_{1}}{ ( 1+{s}/{2} ) }.
\nonumber
\end{eqnarray*}
\end{lemma}

The next result is more delicate; for the sake of brevity, we prove
only (%
\ref{Ubello1}); (\ref{sumLcorol}) can be viewed as a simpler special case
with $L_{1}=1$.
%
\begin{proposition}
\label{sumLintegralscomplete}Let%
\[
Z_{L} ( s ):= \Biggl[ \sum_{l=1}^{L}
( 2l+1 ) l^{1+s}\sum_{l=1}^{L} (
2l+1 ) l^{1+s} ( \log l )^{2}- \Biggl( \sum
_{l=1}^{L} ( 2l+1 ) l^{1+s}\log l
\Biggr)^{2} \Biggr].
\]
Then, for $s\in\mathbb{R}$:%
%
\begin{equation}
\label{sumLcorol} \lim_{L\rightarrow\infty}\frac{1}{L^{4+2s}}Z_{L} ( s ) =
\frac
{1}{%
4 ( 1+{s}/{2} )^{4}}.
\end{equation}
Moreover, let $L_{1}=1+L\cdot ( 1-g ( L )  ) $,
where $%
0<g ( L ) <1$ is such that $\lim_{L\rightarrow\infty}g (
L ) =0$. If
%
\begin{eqnarray}
\label{Zdef} Z_{L;g ( L ) } ( s ) &=&\sum_{l=L_{1}}^{L}(2l+1)l^{1+s}
\sum_{l=L_{1}}^{L}(2l+1)l^{1+s} \bigl(
\log^{2}l \bigr) \nonumber\\[-8pt]\\[-8pt]
&&{}- \Biggl( \sum_{l=L_{1}}^{L}(2l+1)l^{1+s}
\log l \Biggr)^{2},\nonumber
\end{eqnarray}
we have
%
\begin{equation}
\label{Ubello1} \lim_{L\rightarrow\infty}\frac{1}{L^{4 ( 1+{s}/{2} )
}g^{4} ( L ) }Z_{L;g ( L ) } ( s ) =K
( s ),
\end{equation}
where%
\[
K ( s ) =\frac{1}{ ( 1+{s}/{2} )^{2}} \biggl( \frac
{1}{%
12}s^{2}-
\frac{1}{8}s+\frac{1}{3} \biggr).
\]
\end{proposition}

Note that for $s=0$,
%
\begin{equation}
\label{K0} K_{0}=K ( s )\vert_{s=0}=\tfrac{1}{3}.
\end{equation}
\begin{pf*}{Proof of Proposition \ref{sumLintegralscomplete}}
We start by observing that
\begin{eqnarray*}
&& \Biggl( \sum_{l=L_{1}}^{L}(2l+1)l^{s}
\log^{2}l \Biggr) \Biggl( \sum_{l=L_{1}}^{L}(2l+1)l^{s}
\Biggr)
\\
&&\quad=\frac{ ( L^{2 ( 1+{s}/{2} ) }-L_{1}^{2(1+{s}/{2})} )^{2}}{ (
1+{s}/{2} )^{2}} \biggl( \frac{1}{ 
2 ( 1+{s}/{2} )^{2}}+\frac{\log L}{ ( 1+{s/2} )
}+
\log^{2}L \biggr)
\\
&&\qquad{}+\frac{ ( L^{2(1+{s}/{2})}-L_{1}^{2(1+{s}/{2})} ) L_{1}^{2(1+{s}/{2})}}{ ( 1+
{s/2}
)^{3}}\log \bigl( 1-g ( L ) \bigr) \\
&&\qquad\hspace*{11pt}{}\times\biggl(
\frac
{1}{ ( 1+%
{s/2} ) }-2\log L-\log^{2} \bigl( 1-g ( L ) \bigr) \biggr) +
\RMo_{L} ( 1 );
\\
&&\Biggl( \sum_{l=L_{1}}^{L}(2l+1)l^{s}
\log l \Biggr)^{2}
\\
&&\quad=\frac{ ( L^{2(1+{s}/{2})}-L_{1}^{2(1+{s}/{2})} )^{2}}{ (
1+{s/2} )^{2}} \biggl( \frac
{1}{%
4 ( 1+{s/2} )^{2}}-\frac{\log L}{ ( 1+{s/2} )
}+
\log^{2}L \biggr)
\\
&&\qquad{}+\frac{ ( L^{2(1+{s}/{2})}-L_{1}^{2(1+{s}/{2})} ) }{ (
1+{s/2} )^{2}}L_{1}^{2(1+{s}/{2})}\log \bigl( 1-g ( L ) \bigr)
\biggl( \frac
{1}{ ( 1+%
{s/2} ) }-2\log L \biggr)
\\
&&\qquad{}+\frac{L_{1}^{4(1+{s}/{2})}\log^{2} ( 1-g (
L )  ) }{ ( 1+{s/2} )^{2}}+\RMo_{L} ( 1 ),
\end{eqnarray*}
so we obtain%
\begin{eqnarray*}
Z_{L,g ( L ) } ( s ) &=&\frac{ ( L^{2(1+{s}/{2})}-L_{1}^{2(1+{s}/{2})}
)^{2}}{4 (
1+%
{s/2} )^{4}}\\
&&{}-\frac{L^{2(1+{s}/{2})}L_{1}^{2(1+{s}/{2})}\log^{2} (
1-g ( L ) ) }{ ( 1+{s/2} )^{2}}+ \RMo_{L} ( 1 )
\\
&=&\frac{L^{4(1+{s}/{2})} (  ( 1- (
1-g (
L )  )^{2(1+{s}/{2})} )  )^{2}}{%
4 ( 1+{s/2} )^{4}}
\\
&&{}-\frac{L^{4(1+{s}/{2})} ( 1-g ( L )
)^{2(1+{s}/{2})}\log^{2} ( 1-g ( L )
) }{%
( 1+{s/2} )^{2}}+\RMo_{L} ( 1 ).
\end{eqnarray*}
Observe that%
\begin{eqnarray*}
\log^{2}\bigl(1-g ( L) \bigr) &=& \bigl( -g ( L ) -
\tfrac{1}{2}%
g^{2} ( L ) -\tfrac{1}{3}g^{3}
( L ) +\RMO \bigl( g^{4} ( L ) \bigr) \bigr)^{2}
\\
&=&g^{2} ( L ) +g^{3} ( L ) + \bigl( \tfrac
{11}{12}
\bigr) g^{4} ( L ) +\RMo \bigl( g^{4} ( L ) \bigr),
\end{eqnarray*}
while%
%
\begin{eqnarray}
\label{venerdi} \frac{ ( 1-g ( L )  )^{2(1+{s}/{2})}}{%
( 1+{s/2} ) } &=&\frac{1}{ ( 1+{s/2} ) }%
-2g ( L ) +
\biggl( 2 \biggl( 1+\frac{s}{2} \biggr) -1 \biggr) g^{2} ( L )
\nonumber\\[-8pt]\\[-8pt]
&&{}-\frac{ ( 2 ( 1+{s/2} ) -1 )  ( 2 ( 1+
{s/2} ) -2 ) }{3}g^{3} ( L ) +\RMo \bigl( g^{3} ( L )
\bigr).\nonumber
\end{eqnarray}

Thus%
\begin{eqnarray*}
&& \frac{L^{4(1+{s}/{2})} (  ( 1- ( 1-g (
L )  )^{2(1+{s}/{2})} )  )^{2}}{%
4 ( 1+{s/2} )^{4}}
\\
&&\quad=\frac{L^{4(1+{s}/{2})}g^{2} ( L ) }{ ( 1+
{s/2} )^{2}} \bigl[ 1+ ( s+1 ) g ( L ) +\tfrac{1%
}{4} (
s+1 ) \bigl( \tfrac{7}{3}s+1 \bigr) g^{2} ( L ) 
\bigr] +\RMo \bigl( L^{4}g^{4} ( L ) \bigr),
\end{eqnarray*}
while simple calculations lead to%
\begin{eqnarray*}
&& \frac{L^{4(1+{s}/{2})} ( 1-g ( L )  )^{2(1+{s}/{2})}\log^{2} ( 1-g ( L )
) }{%
( 1+{s/2} )^{2}}
\\
&&\quad=\frac{L^{4(1+{s}/{2})}g^{2} ( L ) }{ ( 1+
{s/2} )^{2}} \biggl( 1+ ( s+1 ) g ( L ) + \biggl(
\frac{s^{2}}{2}+\frac{23}{24}s-\frac{1}{12} \biggr)
g^{2} ( L ) \biggr) \\
&&\qquad{}+\RMo \bigl( L^{4}g^{4} ( L )
\bigr).
\end{eqnarray*}
By using (\ref{Ubello1}), we have%
\[
Z_{L,g ( L ) } ( s ) =\frac{L^{4(1+{s}/{2})}g^{4} ( L ) }{ (
1+{s/2} )^{2}}K ( s ) +\RMo \bigl( L^{4}g^{4} ( L ) \bigr)
\]
as claimed.
\end{pf*}

\subsection*{Asymptotic Gaussianity}

In this subsection, we present the analysis of the fourth-order cumulants.
%
\begin{lemma}
\label{cumulants}Let $A_{l}$ and $B_{l}$ be defined as in (\ref{Al}) and
(\ref{Bl}). As $L\rightarrow\infty$,
\[
\frac{1}{L^{4}}\operatorname{cum} \biggl\{ \sum_{l_{1}}(A_{l_{1}}+B_{l_{1}}),
\sum_{l_{2}}(A_{l_{2}}+B_{l_{2}}), \sum
_{l_{3}}(A_{l_{3}}+B_{l_{3}}),\sum
_{l_{4}}(A_{l_{4}}+B_{l_{4}}) \biggr\}
=\RMO_{L}\biggl(\frac{\log^{4}L}{L^{2}}\biggr).
\]
\end{lemma}
\begin{pf}
It is readily checked that%
\begin{eqnarray*}
\operatorname{cum} \biggl\{ \frac{\widehat{C}_{l}}{C_{l}},\frac{\widehat
{C}_{l}}{C_{l}},\frac{%
\widehat{C}_{l}}{C_{l}},
\frac{\widehat{C}_{l}}{C_{l}} \biggr\}
&=&\RMO\bigl(l^{-3}\bigr),
\end{eqnarray*}
\begin{eqnarray*}
&&\operatorname{cum} \biggl\{ \frac{\widehat{G}_{L}(\alpha_{0})}{G_{0}},\frac{\widehat
{G}%
_{L}(\alpha_{0})}{G_{0}},\frac{\widehat{G}_{L}(\alpha_{0})}{G_{0}},
\frac{%
\widehat{G}_{L}(\alpha_{0})}{G_{0}} \biggr\}
\\
&&\quad=\frac{1}{L^{8}}\sum_{l}(2l+1)^{4}\operatorname{cum}
\biggl\{ \frac{\widehat
{C}_{l}}{C_{l}}, \frac{\widehat{C}_{l}}{C_{l}},\frac{\widehat{C}_{l}}{C_{l}},
\frac
{\widehat{C}%
_{l}}{C_{l}} \biggr\} =\RMO\bigl(L^{-6}\bigr).
\end{eqnarray*}
The proof can be divided into 5 cases:
\begin{enumerate}
\item
\begin{eqnarray*}
&&\frac{1}{L^{4}}\operatorname{cum} \biggl\{ \sum_{l_{1}}A_{l_{1}},
\sum_{l_{2}}A_{l_{2}},\sum
_{l_{3}}A_{l_{3}}, \sum_{l_{4}}A_{l_{4}}
\biggr\}
\\
&&\quad=\frac{1}{L^{4}}\sum_{l}(2l+1)^{4}
\bigl\{ \log^{4}l \bigr\} \operatorname{cum} \biggl\{ \frac{\widehat{C}_{l}}{C_{l}},
\frac{\widehat{C}_{l}}{C_{l}},\frac
{\widehat{C}%
_{l}}{C_{l}},\frac{\widehat{C}_{l}}{C_{l}} \biggr\}
\\
&&\quad=\RMO\biggl(\frac{1}{L^{4}}\sum_{l}(2l+1)^{2}
\log^{4}l\biggr)=\RMO\biggl(\frac{\log^{4}L}{L^{2}}\biggr);
\end{eqnarray*}

\item
\begin{eqnarray*}
&&\frac{1}{L^{4}}\operatorname{cum} \biggl\{ \sum_{l_{1}}B_{l_{1}},
\sum_{l_{2}}B_{l_{2}},\sum
_{l_{3}}B_{l_{3}}, \sum_{l_{4}}B_{l_{4}}
\biggr\}
\\
&&\quad=\frac{1}{L^{4}} \biggl\{ \sum_{l}(2l+1)
\log l \biggr\}^{4}\operatorname{cum} \biggl\{ \frac{%
\widehat{G}_{L}(\alpha_{0})}{G_{0}},
\frac{\widehat{G}_{L}(\alpha_{0})}{%
G_{0}},\frac{\widehat{G}_{L}(\alpha_{0})}{G_{0}},\frac{\widehat{G}%
_{L}(\alpha_{0})}{G_{0}} \biggr\}
\\
&&\quad=\frac{1}{L^{4}} \biggl\{ \sum_{l}(2l+1)
\log l \biggr\}^{4}\frac
{1}{L^{6}}=\RMO\biggl(%
\frac{\log^{4}L}{L^{2}}\biggr);
\end{eqnarray*}

\item
\begin{eqnarray*}
&& \frac{1}{L^{4}}\operatorname{cum} \biggl\{ \sum_{l_{1}}A_{l_{1}},
\sum_{l_{2}}B_{l_{2}},\sum
_{l_{3}}B_{l_{3}}, \sum_{l_{4}}B_{l_{4}}
\biggr\}
\\
&&\quad=\frac{1}{L^{4}} \biggl\{ \sum_{l_{1}}(2l_{1}+1)
\log l_{1} \biggr\}^{3}\\
&&\qquad{}\times\sum_{l_{2}}(2l_{2}+1)
\{ \log l_{2} \} \operatorname{cum} \biggl\{ \frac{%
\widehat{C}_{l_{2}}}{C_{l_{2}}},\frac{\widehat{G}_{L}(\alpha_{0})}{G_{0}},
\frac{\widehat{G}_{L}(\alpha_{0})}{G_{0}},\frac{\widehat{G}_{L}(\alpha_{0})%
}{G_{0}} \biggr\}
\\
&&\quad=\frac{1}{L^{10}} \biggl\{ \sum_{l_{1}}(2l_{1}+1)
\log l_{1} \biggr\}^{3}\sum_{l_{2}}(2l_{2}+1)
\log l_{2}
\\
&&\qquad{}\times \operatorname{cum} \biggl\{ \frac{\widehat{C}_{l_{2}}}{C_{l_{2}}},\sum
_{l_{3}}(2l_{3}+1)\frac{\widehat{C}_{l_{3}}}{C_{l_{3}}},\sum
_{l_{3}}(2l_{4}+1)\frac{\widehat{C}_{l_{4}}}{C_{l_{4}}},\sum
_{l_{5}}(2l_{5}+1)\frac{\widehat{C}_{l_{5}}}{C_{l_{5}}} \biggr\}
\\
&&\quad=\frac{\log^{3}L}{L^{4}}\sum_{l}(2l+1)^{4}
\{ \log l \} \operatorname{cum} \biggl\{ \frac{\widehat{C}_{l}}{C_{l}},\frac{\widehat{C}_{l}}{C_{l}},
\frac
{\widehat{C}%
_{l}}{C_{l}},\frac{\widehat{C}_{l}}{C_{l}} \biggr\}
\\
&&\quad=\RMO\biggl(\frac{\log^{3}L}{L^{4}}\sum_{l}(2l+1)
\log l\biggr)=\RMO\biggl(\frac{\log^{4}L}{L^{2}}\biggr);
\end{eqnarray*}

\item
\begin{eqnarray*}
&&\frac{1}{L^{4}}\operatorname{cum} \biggl\{ \sum_{l_{1}}A_{l_{1}},
\sum_{l_{2}}A_{l_{2}},\sum
_{l_{3}}B_{l_{3}}, \sum_{l_{4}}B_{l_{4}}
\biggr\}
\\
&&\quad=\frac{1}{L^{4}}\sum_{l}(2l+1)^{2}
\log^{2}l\operatorname{cum} \biggl\{ \frac{\widehat
{C}_{l}%
}{C_{l}},\frac{\widehat{C}_{l}}{C_{l}},\sum
_{l_{3}}(2l_{3}+1)\log l_{3}
\frac{%
\widehat{G}_{L}(\alpha_{0})}{G_{0}},\\
&&\qquad\hspace*{115.5pt}\sum_{l_{3}}(2l_{4}+1)
\log l_{4}\frac{%
\widehat{G}_{L}(\alpha_{0})}{G_{0}} \biggr\}
\\
&&\quad=\frac{1}{L^{8}} \biggl\{ \sum_{l}(2l+1)
\log l \biggr\}^{2}\\
&&\qquad{}\times\sum_{l}(2l+1)^{2}
\bigl\{ \log^{2}l \bigr\} \operatorname{cum} \biggl\{ \frac
{\widehat{C}%
_{l}}{C_{l}},
\frac{\widehat{C}_{l}}{C_{l}},\sum_{l_{3}}(2l_{3}+1)
\frac{%
\widehat{C}_{l_{3}}}{C_{l_{3}}},\sum_{l_{4}}(2l_{4}+1)
\frac{\widehat{C}%
_{l_{4}}}{C_{l_{4}}} \biggr\}
\\
&&\quad=\frac{1}{L^{8}} \biggl\{ \sum_{l}(2l+1)
\log l \biggr\}^{2}\sum_{l}(2l+1)^{4}
\bigl\{ \log^{2}l \bigr\} \operatorname{cum} \biggl\{ \frac
{\widehat{C}%
_{l}}{C_{l}},
\frac{\widehat{C}_{l}}{C_{l}},\frac{\widehat
{C}_{l}}{C_{l}}, \frac{\widehat{C}_{l}}{C_{l}} \biggr\}
\\
&&\quad=\frac{K}{L^{8}} \biggl\{ \sum_{l}(2l+1)
\log l \biggr\}^{2}\sum_{l}(2l+1)
\log^{2}l=\RMO\biggl(\frac{\log^{4}L}{L^{2}}\biggr);
\end{eqnarray*}

\item
\begin{eqnarray*}
&& \frac{1}{L^{4}}\operatorname{cum} \biggl\{ \sum_{l_{1}}A_{l_{1}},
\sum_{l_{2}}A_{l_{2}},\sum
_{l_{3}}A_{l_{3}}, \sum_{l_{4}}B_{l_{4}}
\biggr\}
\\
&&\quad=\frac{1}{L^{4}}\sum_{l}(2l+1)^{3}
\bigl\{ \log^{3}l \bigr\} \operatorname{cum} \biggl\{ \frac{\widehat{C}_{l}}{C_{l}},
\frac{\widehat{C}_{l}}{C_{l}},\frac
{\widehat{C}%
_{l}}{C_{l}},\sum_{l_{1}}(2l_{1}+1)
\log l_{1}\frac{\widehat
{G}_{L}(\alpha_{0})}{G_{0}} \biggr\}
\\
&&\quad=\frac{1}{L^{6}} \biggl\{ \sum_{l_{1}}(2l_{1}+1)
\log l_{1} \biggr\} \sum_{l}(2l+1)^{3}
\bigl\{ \log^{3}l \bigr\} \operatorname{cum} \biggl\{ \frac{\widehat
{C}_{l}%
}{C_{l}},
\frac{\widehat{C}_{l}}{C_{l}},\frac{\widehat
{C}_{l}}{C_{l}},\sum_{l_{2}}(2l_{2}+1)
\frac{\widehat
{C}_{l_{2}}}{C_{l_{2}}} \biggr\}
\\
&&\quad=\frac{1}{L^{6}} \biggl\{ \sum_{l_{1}}(2l_{1}+1)
\log l_{1} \biggr\} \sum_{l}(2l+1)^{4}
\bigl\{ \log^{3}l \bigr\} \operatorname{cum} \biggl\{ \frac{\widehat
{C}_{l}%
}{C_{l}},
\frac{\widehat{C}_{l}}{C_{l}},\frac{\widehat
{C}_{l}}{C_{l}},\frac{%
\widehat{C}_{l}}{C_{l}} \biggr\}
\\
&&\quad=\frac{1}{L^{6}} \biggl\{ \sum_{l_{1}}(2l_{1}+1)
\log l_{1} \biggr\} \sum_{l}(2l+1)
\log^{3}l=\RMO\biggl(\frac{\log^{4}L}{L^{2}}\biggr).
\end{eqnarray*}\upqed
\end{enumerate}
\end{pf}

\subsection*{Estimation with noise}
%
\begin{lemma}
Under Conditions \ref{A1} and \ref{gammacond}, with $0<\alpha_{0}-\gamma<1$, for all $2>\alpha_{0}-\alpha>\varepsilon>0$, as
$L\rightarrow
\infty$, we have%
\[
\sup_{\alpha}\biggl\llvert \log\frac{\widetilde{G}_{k}(\alpha
)}{G_{k}(\alpha
)}\biggr\rrvert =
\RMo_{p}(1).
\]
On the other hand, if $\alpha_{0}-\alpha\geq2$,
\[
\sup_{\alpha}\biggl\llvert \log\frac{\widetilde{G}_{k}(\alpha
)}{G_{k}(\alpha
)}\biggr\rrvert =
\RMO_{p}(1).
\]
\end{lemma}
\begin{pf}
For the sake of brevity, we report only the proof of the case where
$\alpha
-\alpha_{0}>-2$, using simplified parametric version of Condition \ref{A},
that is, we assume that we have exactly%
\[
C_{l}(\vartheta)=C_{l}(G_{0},
\alpha_{0})=G_{0}l^{-\alpha_{0}}.
\]
As for $\widehat{G}_{k}(\alpha)$,
\[
\frac{\widetilde{G}_{k}(\alpha)}{G_{k}(\alpha)}-1=\frac{%
\sum_{l=1}^{L}(2l+1) ( \log^{k}l ) G_{0}l^{\alpha-\alpha
_{0}} \{ {\widetilde{C}_{l}}/({G_{0}l^{-\alpha_{0}}})-1 \}
}{%
\sum_{l=1}^{L}(2l+1) ( \log^{k}l ) G_{0}l^{\alpha-\alpha
_{0}}}.
\]
Fixed $\max (  ( \alpha_{0}-\gamma ) -1/2,0 )
<\beta<%
\frac{1}{2}$, we have, for all $l$:
\begin{eqnarray*}
&& \Pr \biggl( \biggl\llvert \frac{\sum_{l=1}^{L}(2l+1)G_{0}l^{\alpha-\alpha
_{0}} ( \log^{k}l )  \{ {\widetilde{C}_{l}}/({%
G_{0}l^{-\alpha_{0}}})-1 \} }{\sum_{l=1}^{L}(2l+1) ( \log^{k}l ) G_{0}l^{\alpha-\alpha_{0}}}\biggr\rrvert >
\delta_{\varepsilon
} \biggr)
\\
&&\quad\leq\Pr \biggl( \sup_{l}\sqrt{(2l+1)}l^{- ( \alpha_{0}-\gamma
)
}\biggl
\llvert \frac{\widetilde{C}_{l}}{G_{0}l^{-\alpha_{0}}}-1\biggr\rrvert >\delta_{\varepsilon}L^{\beta}
\biggr),
\end{eqnarray*}
because
\[
L^{\beta}\frac{\sum_{l=1}^{L}\sqrt{(2l+1)} ( \log l )^{k}l^{ ( \alpha-\gamma ) }}{\sum_{l=1}^{L}(2l+1) ( \log
l )^{k}l^{\alpha-\alpha_{0}}}=CL^{\beta-1/2+ ( \alpha
_{0}-\gamma ) }=\RMo(1).
\]

Now%
\begin{eqnarray*}
&&
\Pr \biggl\{ \sup_{l}\sqrt{(2l+1)}l^{- ( \alpha_{0}-\gamma )
}\biggl\llvert
\frac{\widetilde{C}_{l}}{G_{0}l^{-\alpha_{0}}}-1\biggr\rrvert >\delta_{\varepsilon}L^{\beta} \biggr
\} \\
&&\quad\leq L\max_{l}\Pr \biggl\{ \sqrt{%
(2l+1)}l^{- ( \alpha_{0}-\gamma ) }
\biggl\llvert \frac
{\widetilde{C}%
_{l}}{G_{0}l^{-\alpha_{0}}}-1\biggr\rrvert >\delta_{\varepsilon
}L^{\beta
}
\biggr\}
\end{eqnarray*}
and%
\begin{eqnarray*}
&&
\Pr \biggl\{ \sqrt{(2l+1)}l^{- ( \alpha_{0}-\gamma ) } \biggl\llvert
\frac{\widetilde{C}_{l}}{G_{0}l^{-\alpha_{0}}}-1
\biggr\rrvert >\delta_{\varepsilon}L^{\beta} \biggr\} \\
&&\quad\leq C
\frac{\operatorname{Var} [ \sqrt{(2l+1)}%
l^{- ( \alpha_{0}-\gamma ) } ( {\widetilde{C}_{l}}/({%
G_{0}l^{-\alpha_{0}}})-1 )  ] }{\delta_{\varepsilon
}^{2}L^{2\beta}}
\\
&&\quad=\RMO\bigl(L^{-2\beta}\bigr),
\end{eqnarray*}
uniformly in $l$. Hence,
\[
\Pr \biggl\{ \sup_{l}\sqrt{(2l+1)}\biggl\llvert \frac{\widetilde{C}_{l}}{%
G_{0}l^{-\alpha_{0}}}-1
\biggr\rrvert >\delta_{\varepsilon}L^{\beta
} \biggr\} =\RMO
\bigl(L^{-2\beta+1}\bigr)=\RMo_{L}(1).
\]
\upqed
\end{pf}
%
\begin{lemma}
\label{consistency3} Under Conditions \ref{A1} and \ref{gammacond},
with $%
0<\alpha_{0}-\gamma<1$, as $L\rightarrow\infty$, we have%
\begin{eqnarray*}
\sup_{\alpha}\bigl\llvert T_{L} ( \alpha,\alpha_{0}
) \bigr\rrvert &=&\RMo_{p}(1)\qquad\mbox{for }\alpha_{0}-
\alpha<2,
\\
\sup_{\alpha}\bigl\llvert T_{L} ( \alpha,\alpha_{0}
) \bigr\rrvert &=&\RMO_{p}(1)\qquad\mbox{for }\alpha_{0}-
\alpha\geq2.
\end{eqnarray*}
\end{lemma}
\begin{pf}
For $\alpha_{0}-\alpha<2$, consider first%
\[
\frac{\widetilde{G}(\alpha_{0})}{G(\alpha_{0})}-1=\frac{1}{%
\sum_{l=1}^{L}(2l+1)}\sum_{l=1}^{L}(2l+1)
\biggl(\frac{\widetilde{C}_{l}}{%
G_{0}l^{-\alpha_{0}}}-1\biggr),
\]
where we have easily, as $L\rightarrow\infty$,
\begin{eqnarray*}
\mathbb{E} \biggl\{ \frac{\widetilde{G}(\alpha_{0})}{G(\alpha_{0})}%
-1 \biggr\} &=&
\frac{1}{\sum_{l=1}^{L}(2l+1)}\sum_{l=1}^{L}(2l+1)
\biggl(\frac{%
G_{0}l^{-\alpha_{0}} \{ 1+\RMO(l^{-1}) \} }{G_{0}l^{-\alpha
_{0}}}%
-1\biggr)\rightarrow0,
\\
\operatorname{Var} \biggl\{ \frac{\widetilde{G}(\alpha_{0})}{G(\alpha_{0})} \biggr\} &=& \biggl\{ \frac{1}{\sum_{l=1}^{L}(2l+1)} \biggr
\}^{2}2G_{N}^{2}\sum_{l=1}^{L}(2l+1)
\bigl( l^{2 ( \alpha_{0}-\gamma
) }+\RMO \bigl( l^{-\min ( 2\alpha_{0}, ( \gamma+\alpha
_{0} )  ) } \bigr) \bigr)
\\
&=&\RMO\biggl(\frac{1}{L^{4}}L^{2 ( 1+ ( \alpha_{0}-\gamma )
)
}\biggr)\RMO\biggl(
\frac{1}{L^{2 ( 1- ( \alpha_{0}-\gamma )  ) }}\biggr),
\end{eqnarray*}
whence by Slutzky's lemma%
\[
\biggl\{ \frac{\widetilde{G}(\alpha_{0})}{G(\alpha_{0})}\stackrel {\mathbb{P}}%
{\longrightarrow}1
\biggr\} \Rightarrow \biggl\{ \log\frac{\widehat{G}%
(\alpha_{0})}{G(\alpha_{0})}\stackrel{\mathbb{P}} {
\longrightarrow}%
0 \biggr\}.
\]
On the other hand, in view of Lemma \ref{lemmagiovedi}, we have that:%
\[
\sup_{\alpha}\biggl\llvert \log\frac{\widetilde{G}(\alpha_{0})}{G(\alpha_{0})}\biggr\rrvert =
\RMo_{p}(1),
\]
whence the result follows easily. The proof for $\alpha_{0}-\alpha
\geq2$
is immediate.

It remains to prove the consistency of $\widetilde{G}(\widetilde{\alpha
}%
_{L})$. Observe that
\begin{eqnarray*}
\widetilde{G}(\widetilde{\alpha}_{L})-G_{0} &=&
\frac{1}{\sum_{l=1}^{L}(2l+1)%
}\sum_{l=1}^{L}(2l+1)
\frac{\widetilde{C}_{l}}{l^{-\widetilde{\alpha
}_{L}}}-%
\frac{1}{\sum_{l=1}^{L}(2l+1)}\sum
_{l=1}^{L}(2l+1)\frac{G_{0}l^{-\alpha
_{0}}%
}{l^{-\alpha_{0}}}
\\
&=&\frac{1}{\sum_{l=1}^{L}(2l+1)}\sum_{l=1}^{L}(2l+1)G_{0}l^{- (
\alpha_{0}-\widetilde{\alpha}_{L} ) }
\biggl\{ \biggl( \frac
{\widetilde{C%
}_{l}}{G_{0}l^{-\alpha_{0}}}-1 \biggr) + \bigl( 1-l^{ ( \alpha_{0}-%
\widetilde{\alpha}_{L} ) }
\bigr) \biggr\}.
\end{eqnarray*}
Clearly%
\begin{eqnarray*}
\bigl\llvert \widetilde{G}(\widetilde{\alpha}_{L})-G_{0}
\bigr\rrvert &\leq &\Biggl\llvert \frac{1}{\sum_{l=1}^{L}(2l+1)}%
\sum
_{l=1}^{L}(2l+1)G_{0}l^{- ( \alpha_{0}-\widetilde{\alpha}%
_{L} ) }
\biggl\{ \biggl( \frac{\widetilde{C}_{l}}{G_{0}l^{-\alpha
_{0}}}%
-1 \biggr) \biggr\} \Biggr\rrvert
\\
&&{}+\Biggl\llvert \frac{G_{0}}{\sum_{l=1}^{L}(2l+1)}\sum_{l=1}^{L}(2l+1)%
\bigl( 1-l^{ ( \alpha_{0}-\widetilde{\alpha}_{L} ) } \bigr) \Biggr\rrvert =\llvert G_{A}\rrvert
+\llvert G_{B}\rrvert ,
\end{eqnarray*}
so that
\[
\Pr \bigl( \bigl\llvert \widetilde{G}(\widetilde{\alpha}_{L})-G_{0}
\bigr\rrvert \geq\varepsilon \bigr) \leq\Pr \biggl( \llvert G_{A}
\rrvert \geq\frac{\varepsilon}{2} \biggr) +\Pr \biggl( \llvert G_{B}
\rrvert \geq\frac{\varepsilon}{2} \biggr).
\]
Observe that:%
\begin{eqnarray*}
\Pr \biggl\{ \llvert G_{A}\rrvert \geq\frac{\varepsilon
}{2} \biggr\} &
\leq&\Pr \biggl\{ \biggl[ \llvert G_{A}\rrvert \geq\frac
{\varepsilon}{%
2}
\biggr] \cap \biggl[ \llvert \alpha_{0}-\widetilde{
\alpha}%
_{L}\rrvert <\frac{1}{3} \biggr] \biggr\} +
\Pr \biggl\{ \llvert \alpha_{0}-\widetilde{\alpha}_{L}
\rrvert \geq\frac{1}{3} \biggr\}
\\
&\leq&\Pr \Biggl\{ \Biggl[ \frac{1}{\sum_{l=1}^{L}(2l+1)}%
\sum
_{l=1}^{L}(2l+1)G_{0}l^{1/3}
\biggl\llvert \frac{\widetilde{C}_{l}}{%
G_{0}l^{-\alpha_{0}}}-1\biggr\rrvert \geq\varepsilon \Biggr] \Biggr\}
+\RMo_{L}(1)
\\
&\leq&\frac{1}{\varepsilon^{2}}\frac{1}{ ( \sum_{l=1}^{L}(2l+1) )^{2}}\sum_{l=1}^{L}(2l+1)^{2}G_{0}^{2}l^{2/3}\operatorname{Var}
\biggl( \frac{\widehat
{C}_{l}%
}{G_{0}l^{-\alpha_{0}}}-1 \biggr) +\RMo_{L}(1)
\\
&=&\RMO \biggl( \frac{L^{8/3+2 ( \alpha_{0}-\gamma )
}}{L^{4}} \biggr) =\RMo_{L}(1).
\end{eqnarray*}
As far as the second term is concerned, we have, for a suitably small $%
\delta>0$:%
\begin{eqnarray*}
\Pr \biggl( \llvert G_{B}\rrvert \geq\frac{\varepsilon}{2} \biggr) &=&
\Pr \biggl( \biggl[ \llvert G_{B}\rrvert \geq\frac{\varepsilon
}{2}%
\biggr] \cap \bigl[ \log l ( \alpha_{0}-\widetilde{\alpha
}_{L} ) %
\bigr] <\delta \biggr) +\Pr \bigl( \log l (
\alpha_{0}-\widetilde{ 
\alpha}_{L} ) \geq\delta
\bigr)
\\
&=&\Pr \biggl( \biggl[ \llvert G_{B}\rrvert \geq
\frac{\varepsilon
}{2}%
\biggr] \cap \bigl[ \log l ( \alpha_{0}-
\widetilde{\alpha }_{L} ) %
\bigr] <\delta \biggr) +
\RMo_{L}(1)
\end{eqnarray*}
and using $\llvert  \RMe^{-x}-1\rrvert \leq x$ for $0\leq x\leq1$, we
obtain%
\begin{eqnarray*}
\bigl\llvert l^{- ( \alpha_{0}-\widetilde{\alpha}_{L} )
}-1\bigr\rrvert &=&\bigl\llvert \exp\bigl(-\log l (
\alpha_{0}-\widetilde {\alpha }_{L} )\bigr) -1\bigr\rrvert \leq 
\log l\llvert \alpha_{0}-\widetilde{ 
\alpha}_{L}
\rrvert ,\\[-27pt]
\end{eqnarray*}
\begin{eqnarray*}
&&\Pr \biggl( \biggl[ \llvert G_{B}\rrvert \geq\frac{\varepsilon
}{2}%
\biggr] \cap \bigl[ \log l ( \alpha_{0}-\widetilde{\alpha
}_{L} ) %
\bigr] <\delta \biggr)
\\
&&\quad\leq\Pr \Biggl( \frac{G_{0}}{\sum_{l=1}^{L}(2l+1)}\sum_{l=1}^{L}(2l+1)
\bigl\llvert \bigl( l^{- ( \alpha_{0}-\widetilde{\alpha}_{L} )
}-1 \bigr) \bigr\rrvert \geq
\frac{\varepsilon}{2}\cap \bigl[ \log l ( \alpha_{0}-\widetilde{
\alpha}_{L} ) \bigr] <\delta \Biggr)
\\
&&\quad\leq\frac{1}{\varepsilon^{2}}\operatorname{Var} \Biggl\{ \frac{G_{0}}{\sum_{l=1}^{L}(2l+1)%
}\sum
_{l=1}^{L}(2l+1)\log l\llvert \alpha_{0}-
\widetilde{\alpha}%
_{L}\rrvert \Biggr\}
\\
&&\quad=\RMO \biggl( \frac{1}{L^{4}}L^{2}\log L\frac{1}{L^{2-2 ( \alpha
_{0}-\gamma ) }}
\biggr) =\RMo_{L}(1),
\end{eqnarray*}
where we have used
\[
\operatorname{Var} ( \alpha_{0}-\widetilde{\alpha}_{L} ) =\RMO\biggl(
\frac{1}{%
L^{2-2 ( \alpha_{0}-\gamma ) }}\biggr),
\]
which under Condition \ref{A1} will be established in the proof of
Theorem %
\ref{theonoise}.
\end{pf}

Finally, we provide the proof of the central limit theorem in the
presence of
observational noise.
\begin{pf*}{Proof of Theorem \ref{theonoise}}
The main difference with the argument in the
noiseless case concerns the variance of the score $\overline
{S}_{L}(\alpha_{0})$; we just sketch the main steps and leave the details to the reader.
Indeed, we can split $\operatorname{Var} \{ \overline{S}_{L}(\alpha_{0}) \}
$ as
\[
\operatorname{Var} \bigl\{ \overline{S}_{L}(\alpha_{0}) \bigr\}
=V_{1}+V_{2}+V_{3},
\]
where%
\begin{eqnarray*}
V_{1}&=& \biggl\{ \frac{1}{\sum_{l=1}^{L}(2l+1)} \biggr\}^{2}\sum
_{l=1}^{L} ( 2l+1 )^{2} ( \log l
)^{2}\operatorname{Var} \biggl\{ \frac{\widetilde{C}_{l}}{G_{0}l^{-\alpha_{0}}} \biggr\},
\\
V_{2}&=& \biggl\{ \frac{1}{\sum_{l=1}^{L}(2l+1)} \biggr\}^{2} \Biggl( \sum
_{l=1}^{L} ( 2l+1 ) \log l
\Biggr)^{2}\operatorname{Var} \biggl( \frac{%
\widetilde{G}(\alpha_{0})}{G_{0}} \biggr),
\\
V_{3}&=&\frac{-2 ( \sum_{l=1}^{L} ( 2l+1 ) \log l )
}{ (
\sum_{l=1}^{L}(2l+1) )^{2}}\cdot\sum_{l=1}^{L}
\bigl\{ ( 2l+1 ) \log l \bigr\} \operatorname{Cov} \biggl( \frac{\widetilde
{C}_{l}}{C_{l}},
\frac{%
\widetilde{G}(\alpha_{0})}{G_{0}} \biggr).
\end{eqnarray*}
Here%
%
\begingroup
\abovedisplayskip=7pt
\belowdisplayskip=7pt
\begin{eqnarray}
\label{varGgen} \operatorname{Var}
\biggl( \frac{\widehat{G}(\alpha_{0})}{G_{0}} \biggr) &=&\frac{2}{%
\sum_{l=1}^{L}(2l+1)}\nonumber\\[-2pt]
&&\hspace*{0pt}{}\times
\biggl( 1+ \biggl( \frac{G_{N}}{G_{0}} \biggr)^{2}\frac{%
\sum_{l=1}^{L} ( 2l+1 ) l^{-2 ( \gamma-\alpha_{0} )
}}{%
\sum_{l=1}^{L} ( 2l+1 ) }\nonumber\\[-9pt]\\[-9pt]
&&\qquad\hspace*{0pt}{}+
\biggl( \frac{G_{N}}{G_{0}} \biggr) \frac{%
\sum_{l=1}^{L} ( 2l+1 ) l^{- ( \gamma-\alpha_{0} )
}}{%
\sum_{l=1}^{L} ( 2l+1 ) } \biggr)
\nonumber\\[-2pt]
&&{}+\RMO \bigl( L^{-\min ( 2 ( \gamma-\alpha_{0} ), (
\gamma
-\alpha_{0} )  ) -2} \bigr),
\nonumber
\\[-2pt]
\label{CovGgen} \operatorname{Cov}
\biggl( \frac{\widetilde{C}_{l}}{C_{l}^{T}},\frac{\widehat{G}(\alpha_{0})%
}{G_{0}} \biggr) &=&
\frac{2}{\sum_{l=1}^{L}(2l+1)} \nonumber\\[-2pt]
&&{}\times\biggl( 1+ \biggl( \frac
{G_{N}}{%
G_{0}} \biggr)^{2}l^{-2 ( \gamma-\alpha_{0} ) }\\[-2pt]
&&\qquad{}+2
\frac
{G_{N}}{%
G_{0}}l^{- ( \gamma-\alpha_{0} ) }+\RMO \bigl( l^{-\min (
2 ( \gamma-\alpha_{0} ), ( \gamma-\alpha_{0} )
) } \bigr) \biggr);\nonumber
\end{eqnarray}
hence
\begin{eqnarray*}
V_{1} &=& \biggl( \frac{1}{\sum_{l=1}^{L}(2l+1)} \biggr)^{2}\\[-2pt]
&&{}\times 2\sum
_{l=1}^{L} ( 2l+1 ) ( \log l
)^{2}
\biggl( 1+ \biggl( \frac{G_{N}}{G_{0}} \biggr)^{2}l^{-2 (
\gamma
-\alpha_{0} ) }\\[-2pt]
&&\qquad\quad\hspace*{69pt}{}+2
\frac{G_{N}}{G_{0}}l^{- ( \gamma+\alpha
_{0} ) }+\RMo \bigl( l^{-\min ( 2 ( \gamma-\alpha_{0}
), ( \gamma-\alpha_{0} )  ) } \bigr) \biggr);
\\[-2pt]
V_{2}&=& \biggl( \frac{1}{\sum_{l=1}^{L}(2l+1)} \biggr)^{3}2 \Biggl( \sum
_{l=1}^{L} ( 2l+1 ) \log l
\Biggr)^{2}
\\[-2pt]
&&{}\times \biggl( 1+ \biggl( \frac{G_{N}}{G_{0}} \biggr)^{2}
\frac{%
\sum_{l=1}^{L} ( 2l+1 ) l^{-2 ( \gamma-\alpha_{0} )
}}{%
\sum_{l=1}^{L} ( 2l+1 ) }+ \biggl( \frac{G_{N}}{G_{0}} \biggr) \frac{%
\sum_{l=1}^{L} ( 2l+1 ) l^{- ( \gamma-\alpha_{0} )
}}{%
\sum_{l=1}^{L} ( 2l+1 ) }
\biggr) \\[-2pt]
&&{}+\RMo \bigl( L^{-\min (
2 (
\gamma-\alpha_{0} ), ( \gamma-\alpha_{0} )  )
} \bigr);
\\[-2pt]
V_{3} &=&\frac{-4 ( \sum_{l=1}^{L} ( 2l+1 ) \log l )
}{%
( \sum_{l=1}^{L}(2l+1) )^{3}}\sum_{l=1}^{L}
( 2l+1 ) \log l
\\[-2pt]
&&\hspace*{111.4pt}{}\times \biggl( 1+ \biggl( \frac{G_{N}}{G_{0}} \biggr)^{2}l^{-2 (
\gamma-\alpha_{0} ) }+2
\frac{G_{N}}{G_{0}}l^{- (
\gamma+\alpha_{0} ) }\\[-2pt]
&&\hspace*{129pt}{}+\RMO \bigl( l^{-\min ( 2 (
\gamma-\alpha_{0} ),
(\gamma-\alpha_{0} )  )} \bigr) \biggr).
\end{eqnarray*}
\endgroup
For $\gamma\geq\alpha_{0}$, we have hence
%
\begin{equation}
\label{varSl1} \lim_{L\rightarrow\infty}2 \biggl( 1+\frac{G_{N}}{G_{0}}
\delta_{\alpha
_{0}}^{\gamma} \biggr)^{2}L^{2}\operatorname{Var}
\bigl\{ \overline{S}_{L}(\alpha_{0}) \bigr\} =1.
\end{equation}
In fact, for $\alpha_{0}<\gamma$, we obtain%
\begin{eqnarray*}
V_{1}&=& \biggl( \frac{1}{\sum_{l=1}^{L}(2l+1)} \biggr)^{2}2\sum
_{l=1}^{L} ( 2l+1 ) ( \log l )^{2} \bigl(
1+\RMO \bigl( l^{- ( \gamma
-\alpha_{0} ) } \bigr) \bigr);
\\
V_{2}&=& \biggl\{ \frac{1}{\sum_{l=1}^{L}(2l+1)} \biggr\}^{3}2 \Biggl(
\sum_{l=1}^{L} ( 2l+1 ) \log l
\Biggr)^{2}+\RMO \bigl( L^{- (
\gamma-\alpha_{0} ) -2} \bigr);
\\
V_{3}&=&-4 \biggl\{ \frac{1}{\sum_{l=1}^{L}(2l+1)} \biggr\}^{3} \Biggl(
\sum_{l=1}^{L} ( 2l+1 ) \log l
\Biggr)^{2}+\RMO \bigl( L^{- (
\gamma-\alpha_{0} ) -2} \bigr),
\end{eqnarray*}
so that
\begin{eqnarray*}
&& \operatorname{Var} \bigl\{ \overline{S}_{L}(\alpha_{0}) \bigr\}
\\
&&\quad=\frac{2}{ ( \sum_{l=1}^{L}(2l+1) )^{3}} \Biggl( \sum_{l=1}^{L}
( 2l+1 ) \sum_{l=1}^{L} ( 2l+1 ) ( \log l
)^{2}- \Biggl( \sum_{l=1}^{L} (
2l+1 ) \log l \Biggr)^{2} \Biggr) \\
&&\qquad{}+\RMO \bigl( L^{- ( \gamma-\alpha_{0} ) -2} \bigr)
\\
&&\quad=\frac{2}{L^{6}}\frac{L^{4}}{4}+\RMO \bigl( L^{- ( \gamma-\alpha
_{0} ) -2}
\bigr) =\frac{1}{2L^{2}}+\RMO \bigl( L^{- ( \gamma
-\alpha
_{0} ) -2} \bigr)
\end{eqnarray*}
by using (\ref{sumLcorol}) and (\ref{slog0}) with $s=0$ to obtain (\ref%
{varSl1}). Similarly, if $\alpha_{0}=\gamma$, we have%
\begin{eqnarray*}
\operatorname{Var} \biggl( \frac{\widehat{G}(\alpha_{0})}{G_{0}} \biggr) &=&\frac{2}{%
\sum_{l=1}^{L}(2l+1)} \biggl( 1+
\frac{G_{N}}{G_{0}} \biggr)^{2}+\RMO \bigl( L^{- ( \gamma-\alpha_{0} ) -2} \bigr);
\\
\operatorname{Cov} \biggl( \frac{\widetilde{C}_{l}}{C_{l}^{T}},\frac{\widehat{G}(\alpha_{0})%
}{G_{0}} \biggr) &=&\frac{2}{\sum_{l^{\prime}=1}^{L}(2l+1)}
\biggl( 1+\frac
{G_{N}%
}{G_{0}} \biggr)^{2}+\RMO \bigl( L^{- ( \gamma-\alpha_{0} )
-2}
\bigr).
\end{eqnarray*}
Simple calculations lead then to (\ref{varSl1}). For $\gamma<\alpha_{0}<\gamma+1$, we have%
\begin{eqnarray*}
V_{1} &=&\frac{2 ( {G_{N}}/{G_{0}} )^{2}}{ (
\sum_{l=1}^{L}(2l+1) )^{4}} \Biggl( \sum
_{l=1}^{L}(2l+1) \Biggr)^{2}\sum
_{l=1}^{L} ( 2l+1 ) ( \log l )^{2} \bigl(
l^{2 ( \alpha_{0}-\gamma ) }+\RMo \bigl( l^{2 ( \alpha
_{0}-\gamma ) } \bigr) \bigr)
\\
&=&\frac{2 ( {G_{N}}/{G_{0}} )^{2}}{ (
\sum_{l=1}^{L}(2l+1) )^{4}}\frac{L^{6+2 ( \alpha_{0}-\gamma
) }}{1+ ( \alpha_{0}-\gamma ) } \\
&&{}\times\biggl( \log^{2}L-
\frac
{\log
L}{ ( 1+ ( \alpha_{0}-\gamma )  ) }+\frac{L^{2 (
1+ ( \alpha_{0}-\gamma )  ) }}{ ( 1+ ( \alpha_{0}-\gamma )  )^{2}}+\RMo(1) \biggr);
\\
V_{2} &=&\frac{2 ( {G_{N}}/{G_{0}} )^{2}}{ (
\sum_{l=1}^{L}(2l+1) )^{4}} \Biggl( \sum
_{l=1}^{L} ( 2l+1 ) \log l \Biggr)^{2}\sum
_{l=1}^{L} ( 2l+1 ) \bigl(
l^{2 (
\alpha
_{0}-\gamma ) }+\RMo \bigl( l^{2 ( \alpha_{0}-\gamma )
} \bigr) \bigr)
\\
&=&\frac{2 ( {G_{N}}/{G_{0}} )^{2}}{ (
\sum_{l=1}^{L}(2l+1) )^{4}}\frac{L^{6+2 ( \alpha_{0}-\gamma
) }}{1+ ( \alpha_{0}-\gamma ) } \biggl( \log^{2}L-\log
L+%
\frac{1}{4}+\RMo(1) \biggr);
\\
V_{3} &=&\frac{-4 ( {G_{N}}/{G_{0}} )^{2}}{ (
\sum_{l=1}^{L}(2l+1) )^{4}} \Biggl( \sum
_{l=1}^{L}(2l+1) \Biggr) \Biggl( \sum
_{l=1}^{L} ( 2l+1 ) \log l \Biggr)
\\
&&{}\times \Biggl( \sum_{l=1}^{L} ( 2l+1
) \log l \biggl( \frac
{G_{N}}{%
G_{0}} \biggr)^{2} \bigl(
l^{2 ( \alpha_{0}-\gamma ) }+\RMo \bigl( l^{2 ( \alpha_{0}-\gamma ) } \bigr) \bigr) \Biggr)
\\
&=&\frac{-4 ( {G_{N}}/{G_{0}} )^{2}}{ (
\sum_{l=1}^{L}(2l+1) )^{4}}\frac{L^{6+2 ( \alpha_{0}-\gamma
) }}{1+ ( \alpha_{0}-\gamma ) }
\\
&&{}\times \biggl( \log^{2}L+\frac{1}{4 ( 1+ ( \alpha_{0}-\gamma
)  ) }-\frac{\log L}{2}
\biggl( 1+\frac{1}{ ( 1+ (
\alpha_{0}-\gamma )  ) } \biggr) +\RMo(1) \biggr)
\end{eqnarray*}
by using (\ref{sumLcorol}) and (\ref{slog0}) with $s=2 ( \alpha_{0}-\gamma )$. Hence, we obtain
\[
\lim_{L\rightarrow\infty}L^{2-2 ( \alpha_{0}-\gamma )
}\operatorname{Var} \bigl\{ \overline{S}_{L}(
\alpha_{0}) \bigr\} =2 \biggl( \frac
{G_{N}}{G_{0}%
}
\biggr)^{2}H ( \alpha_{0}-\gamma ),
\]
so that the asymptotic behaviour of the variance is fully understood.

To conclude the proof of the central limit theorem, let us focus on
$\gamma
<\alpha_{0}<\gamma+1$ and write%
\[
L^{1- ( \alpha_{0}-\gamma ) }S_{L}(\alpha_{0})=\frac{1}{%
L^{1+ ( \alpha_{0}-\gamma ) }+\RMO(L^{1+ ( \alpha
_{0}-\gamma
) })}\sum
_{l}(A_{l}+B_{l}),
\]
where%
\[
A_{l}=(2l+1)\log l \biggl\{ \frac{\widetilde{C}_{l}}{C_{T,l}}-1
\biggr\},\qquad
B_{l}=(2l+1)\log l \biggl\{ \frac{\widetilde{G}_{L}(\alpha_{0})}{G_{0}}%
-1 \biggr
\}.
\]
The analysis of fourth-order cumulants%
\begin{eqnarray*}
&&
\frac{1}{L^{4 ( 1+ ( \alpha_{0}-\gamma )  )
}}\operatorname{cum} \biggl\{ \sum_{l_{1}}(A_{l_{1}}+B_{l_{1}}),
\sum_{l_{2}}(A_{l_{2}}+B_{l_{2}}), \sum
_{l_{3}}(A_{l_{3}}+B_{l_{3}}),\sum
_{l_{4}}(A_{l_{4}}+B_{l_{4}}) \biggr\}\\
&&\quad
=\RMO_{L}\biggl(\frac{\log^{4}L}{L^{2+ ( \alpha_{0}-\gamma ) }}\biggr)
\end{eqnarray*}
is entirely analogous to the noiseless case.
\end{pf*}
\end{appendix}



\printhistory

\end{document}